\documentclass{amsart}
\usepackage{graphicx, color}
\usepackage{amscd}
\usepackage{amsmath,empheq}
\usepackage{amsfonts}
\usepackage{amssymb}
\usepackage{mathrsfs}

\usepackage[all]{xy}

\date{\today}

\newtheorem{theorem}{Theorem}

\newtheorem{lemma}[theorem]{Lemma}
\newtheorem{prop}[theorem]{Proposition}
\newtheorem{remark}{Remark}
\newtheorem{corollary}[theorem]{Corollary}
\newtheorem{claim}{Claim}

\newtheorem{example}{Example}

\newenvironment{proof-sketch}{\noindent{\bf Sketch of Proof}\hspace*{1em}}{\qed\bigskip}

\everymath{\displaystyle}

\newcommand{\RR}{\mathbb R}
\newcommand{\NN}{\mathbb N}

\renewcommand{\leq}{\leqslant}
\renewcommand{\ge}{\geqslant}
\renewcommand{\geq}{\geqslant}

\begin{document}

\title[Nonlinear nonhomogeneous parametric Robin problems]{Positive solutions for nonlinear nonhomogeneous parametric Robin problems}

\author[N.S. Papageorgiou]{Nikolaos S. Papageorgiou}
\address[N.S. Papageorgiou]{National Technical University, Department of Mathematics,
				Zografou Campus, Athens 15780, Greece}
\email{\tt npapg@math.ntua.gr}

\author[V.D. R\u{a}dulescu]{Vicen\c{t}iu D. R\u{a}dulescu}
\address[V.D. R\u{a}dulescu]{Department of Mathematics, Faculty of Sciences,
King Abdulaziz University, P.O. Box 80203, Jeddah 21589, Saudi Arabia  \& Department of Mathematics, University of Craiova, Street A.I. Cuza 13,
          200585 Craiova, Romania}
\email{\tt vicentiu.radulescu@imar.ro}

\author[D.D. Repov\v{s}]{Du\v{s}an D. Repov\v{s}}
\address[D.D. Repov\v{s}]{Faculty of Education and Faculty of Mathematics and Physics, University of Ljubljana, SI-1000 Ljubljana, Slovenia}
\email{\tt dusan.repovs@guest.ames.si}


\begin{abstract}
We study a parametric Robin problem driven by a nonlinear nonhomogeneous differential operator and with a superlinear Carath\'eodory reaction term. We prove a bifurcation-type theorem for small values of the parameter. Also, we show that as the parameter $\lambda>0$ approaches zero we can find positive solutions with arbitrarily big and arbitrarily small Sobolev norm. Finally we show that for every admissible parameter value there is a smallest positive solution $u^*_{\lambda}$ of the problem and we investigate the properties of the map $\lambda\mapsto u^*_{\lambda}$.
\end{abstract}

\maketitle


\section{Introduction}

Let $\Omega\subseteq \RR^N$ be a bounded domain with a $C^2$-boundary $\partial\Omega$. In this paper, we study the following nonlinear, nonhomogeneous parametric Robin problem
\begin{equation}\tag{$P_{\lambda}$}\label{eqp}
	\left\{\begin{array}{ll}
		-{\rm div}\, a(Du(z))+\xi(z)u(z)^{p-1}=\lambda f(z,u(z))\ \mbox{in}\ \Omega,\\
		\frac{\partial u}{\partial n_a}+\beta(z)u^{p-1}=0\ \mbox{on}\ \partial\Omega,\ u> 0,\ \lambda>0,\ 1<p<\infty\,.&
	\end{array}\right\}
\end{equation}

In this problem, the map $a:\RR^N\rightarrow\RR^N$ is monotone continuous (hence maximal monotone, too) and satisfies certain other regularity and growth conditions, listed in hypotheses $H(a)$ below. These conditions on $a(\cdot)$, are general enough to incorporate in our framework many differential operators of interest such as the $p$-Laplacian differential operator $(1<p<\infty)$ and the $(p,q)$-Laplacian differential operator ($1<q<p<\infty$). The differential operator in \eqref{eqp} is not in general $(p-1)$-homogeneous and this is a source of technical difficulties in the analysis of problem \eqref{eqp}. Also $\xi\in L^{\infty}(\Omega)$ and $\xi\geq 0$. In the reaction term (right-hand side of the equation) $\lambda>0$ is a parameter and $f(z,x)$ is a Carath\'eodory function (that is, for all $x\in\RR$, the mapping $z\mapsto f(z,x)$ is measurable and for almost all $z\in\Omega$, the mapping $x\mapsto f(z,x)$ is continuous) which exhibits $(p-1)$-superlinear growth in the $x$-variable near
  $+\infty$, but without satisfying the usual for superlinear problems Ambrosetti-Rabinowitz condition (AR condition for short). Instead we use a more general condition, which permits the consideration of $(p-1)$-superlinear functions with ``slower'' growth near $+\infty$ which fail to satisfy the AR-condition (see the examples below). Also near $0^+$, the nonlinearity $f(z,\cdot)$ has a concave term (that is, a ($p-1$)-sublinear term).

In the boundary condition, $\frac{\partial u}{\partial n_a}$ denotes the generalized normal derivative (the conormal derivative) of $u$, defined by extension of
$$\frac{\partial u}{\partial n_a}=(a(Du),n)_{\RR^N}\ \mbox{for all}\ u\in C^1(\overline{\Omega}),$$
with $n(\cdot)$ being the outward unit normal on $\partial\Omega$. This kind of directional derivative on the boundary $\partial\Omega$ is dictated by the nonlinear Green's identity (see Gasinski and Papageorgiou \cite[p. 210]{14}) and is also used by Lieberman \cite{22}. For the boundary coefficient $\beta(z)$, we assume that
$$\beta\in C^{0,\alpha}(\partial\Omega)\ \mbox{for some}\ \alpha\in(0,1),\ \beta(z)\geq 0\ \mbox{for all}\ z\in\partial\Omega.$$

We assume that
$$\xi\neq 0\ \mbox{or}\ \beta\neq 0.$$

If $\beta=0$, then we recover the Neumann problem.

Our aim in this paper is to study the precise dependence of the set of positive solutions on the parameter $\lambda>0$. In this direction, we prove a bifurcation-type theorem for small values of the parameter, that is, we show that there exists a critical parameter value $\lambda^*\in(0,+\infty)$ such that
\begin{itemize}
	\item for all $\lambda\in(0,\lambda^*)$ problem \eqref{eqp} admits at least two positive solutions;
	\item for $\lambda=\lambda^*$ problem \eqref{eqp} has at least one positive solutions;
	\item for all $\lambda>\lambda^*$ problem \eqref{eqp} has no positive solutions.
\end{itemize}

Moreover, we show that if $\lambda_n\rightarrow 0^+$, then we can find pairs $\{u_{\lambda_n},\hat{u}_{\lambda_n}\}_{n\in\NN}$ of positive solutions such that
$$||u_{\lambda_n}||\rightarrow 0\ \mbox{and}\ ||\hat{u}_{\lambda_n}||\rightarrow+\infty\ \mbox{as}\ n\rightarrow\infty.$$

Here $||\cdot||$ denotes the norm of the Sobolev space $W^{1,p}(\Omega)$.

Finally if $\lambda\in(0,\lambda^*)$, then we show that problem \eqref{eqp} has a smallest positive solution $u^*_{\lambda}$ and we investigate the monotonicity and continuity properties of the map $\lambda\mapsto u^*_{\lambda}$.

Parametric problems with competing nonlinearities (``concave-convex'' problems), were first investigated by Ambrosetti, Brezis and Cerami \cite{3} for semilinear Dirichlet problems driven by the Laplacian (that is, $p=2$) and with zero potential (that is, $\xi\equiv 0$). Their work was extended to Dirichlet problems driven by the $p$-Laplacian $(1<p<\infty)$ by Garcia Azorero, Manfredi and Peral Alonso \cite{13}, Guo and Zhang \cite{18}, Hu and Papageorgiou \cite{20}. All the aforementioned papers, consider ``concave-convex'' reaction terms modelled after the function
$$\lambda x^{q-1}+x^{r-1}\ \mbox{for all}\ x\geq 0,\ \mbox{with}\ q<p<r<p^*.$$

So, in their equations the concave and convex inputs in the reaction are decoupled and the parameter $\lambda>0$ multiplies only the concave term.

Closer to problem \eqref{eqp} are the works of Gasinski and Papageorgiou \cite{16}, Papageorgiou and R\u adulescu \cite{33} and Aizicovici, Papageorgiou and  Staicu \cite{2bis}. Both papers deal with equations driven by the $p$-Laplacian and have a reaction term of the form $\lambda f(z,x)$ (as is the case here). In Gasinski and Papageorgiou \cite{16} the problem is Dirichlet and the authors prove bifurcation-type results for small and big values of the parameter $\lambda>0$. In Papageorgiou and R\u adulescu \cite{33} the problem is Robin (with $\xi\equiv 0,\beta\neq 0$) and the authors prove a bifurcation-type result for large values of the parameter. Finally, we mention also the related recent work of Papageorgiou and Smyrlis \cite{38} who deal with singular Dirichlet problems and of Papageorgiou and R\u adulescu \cite{34} dealing with $p$-Laplacian Robin problems with competing nonlinearities.

We denote by $||\,\cdot\,||_p$ the usual $L^p$-norm in $L^p(\Omega)$ and by $|\,\cdot\,|$ the Euclidean norm on $\RR^N$. Throughout this paper, the symbol $\stackrel{w}{\rightarrow}$ is used for the weak convergence.

\section{Mathematical Background-Auxiliary Results}

Let $X$ be a Banach space and $X^*$ its topological dual. By $\left\langle \cdot,\cdot\right\rangle$ we denote the duality brackets for the pair $(X^*,X)$. If $\varphi\in C^1(X,\RR)$, we say that $\varphi$ satisfies the ``Cerami condition'' (the ``C-condition'' for short), if the following property holds:
\begin{center}
``Every sequence $\{u_n\}_{n\geq 1}\subseteq X$ such that $\{\varphi(u_n)\}_{n\geq 1}\subseteq\RR$ is bounded and
$$(1+||u_n||)\varphi'(u_n)\rightarrow 0\ \mbox{in}\ X^*\ \mbox{as}\ n\rightarrow\infty,$$
admits a strongly convergent subsequence''.
\end{center}

This compactness-type condition on the functional $\varphi$, leads to a deformation theorem from which one can derive the minimax theory of the critical values of $\varphi$. Central in that theory, is the well-known ``mountain pass theorem'' due to Ambrosetti and Rabinowitz \cite{4}, stated here in a slightly more general form (see Gasinski and Papageorgiou \cite[p. 648]{14}).
\begin{theorem}\label{th1}
	Assume that $X$ is a Banach space, $\varphi\in C^1(X,\RR)$ satisfies the C-condition, $u_0,u_1\in X$, $||u_1-u_0||>\rho>0$
	$$\max\{\varphi(u_0),\varphi(u_1)\}<\inf[\varphi(u):||u-u_0||=\rho]=m_{\rho}$$
	and $c=\inf\limits_{\gamma\in\Gamma}\max\limits_{0\leq t\leq 1}\varphi(\gamma(t))$ with $\Gamma=\{\gamma\in C([0,1],X):\gamma(0)=u_0,\gamma(1)=u_1\}$.
	Then $c\geq m_{\rho}$ and $c$ is a critical value of $\varphi$.
\end{theorem}
\begin{remark}
	The result is in fact true more generally in Banach-Finsler manifolds.
\end{remark}

By $||\cdot||$ we denote the norm of the Sobolev space $W^{1,p}(\Omega)$ defined by
$$||u||=[||u||^p_p+||Du||^p_p]^{1/p}\ \mbox{for all}\ u\in W^{1,p}(\Omega).$$

In addition to the Sobolev space $W^{1,p}(\Omega)$ we will also use the Banach space $C^1(\overline{\Omega})$ and certain closed subspaces of it and the ``boundary'' Lebesgue spaces $L^q(\partial\Omega)$ $(1\leq q\leq \infty)$. The space $C^1(\overline{\Omega})$ is an ordered Banach space with positive (order) cone given by
$$C_+=\{u\in C^1(\overline{\Omega}):u(z)\geq 0\ \mbox{for all}\ z\in\overline{\Omega}\}.$$

The cone has a nonempty interior given by
$$D_+=\{u\in C^1(\overline{\Omega}):u(z)>0\ \mbox{for all}\ z\in\overline{\Omega}\}.$$

On $\partial\Omega$ we consider the $(N-1)$-dimensional Hausdorff (surface) measure $\sigma(\cdot)$. Using this measure we can define in the usual way the boundary Lebesgue spaces $L^q(\partial\Omega)$ $(1\leq q\leq\infty)$. From the theory of Sobolev spaces, we know that there exists a unique continuous linear map $\gamma_0:W^{1,p}(\Omega)\rightarrow L^p(\partial\Omega)$ known as the ``trace map'' such that
$$\gamma_0(u)=u|_{\partial\Omega}\ \mbox{for all}\ u\in W^{1,p}(\Omega)\cap C(\overline{\Omega}).$$

We know that
$${\rm im}\,\gamma_0=W^{\frac{1}{p'},p}(\partial\Omega)\ \left(\frac{1}{p}+\frac{1}{p'}=1\right)\ \mbox{and}\ {\rm ker}\,\gamma_0=W^{1,p}_0(\Omega).$$

The trace map $\gamma_0$ is compact into $L^q(\partial\Omega)$ for all $q\in\left[1,\frac{(N-1)p}{N-p}\right)$ if $N>p$ and into $L^q(\partial\Omega)$ for all $q\geq 1$ if $p\geq N$. In the sequel, for the sake of notational simplicity we will drop the use of the map $\gamma_0$. The restrictions of all Sobolev functions on $\partial\Omega$ are understood in the sense of traces.

Let $\vartheta\in C^1(0,+\infty)$ with $\vartheta(t)>0$ for all $t>0$ and assume that
\begin{equation}\label{eq1}
	0<\hat{c}\leq\frac{\vartheta'(t)t}{\vartheta(t)}\leq c_0\ \mbox{and}\ c_1t^{p-1}\leq\vartheta(t)\leq c_2(1+t^{p-1})\ \mbox{for all}\ t>0,\ \mbox{some}\ c_1,c_2>0.
\end{equation}

Our hypotheses on the map $a(\cdot)$, are the following:

\smallskip
$H(a):$ $a(y)=a_0(|y|)y$ for all $y\in \RR^N$ with $a_0(t)>0$ for all $t>0$ and
\begin{itemize}
	\item[(i)] $a_0\in C^1(0,\infty),\ t\mapsto a_0(t)t$ is strictly increasing on $(0,+\infty)$, $a_0(t)t\rightarrow 0^+$ as $t\rightarrow 0^+$ and
	$$\lim\limits_{t\rightarrow 0^+}\frac{a'_0(t)t}{a_0(t)}>-1;$$
	\item[(ii)] there exists $c_3>0$ such that
	$$|\nabla a(y)|\leq c_3\frac{\vartheta(|y|)}{|y|}\ \mbox{for all}\ y\in \RR^N\backslash\{0\};$$
	\item[(iii)] $(\nabla a(y)\xi,\xi)_{\RR^N}\geq\frac{\vartheta(|y|)}{|y|}|\xi|^2$ for all $y\in\RR^N\backslash\{0\}$, all $\xi\in\RR^N$;
	\item[(iv)] if $G_0(t)=\int^t_0a_0(s)sds,$
\end{itemize}
then there exist $1<q<p<r_0<p^*$ (recall that $p^*=
	\frac{Np}{N-p}$ if $N>p$ and $p^*=
	+\infty$ if $N\leq p$) such that
\begin{eqnarray*}
	&&\limsup\limits_{t\rightarrow 0^+}\frac{q G_0(t)}{t^q}\leq c^*,\ t\mapsto G_0(t^{1/q})\ \mbox{is convex}\\
	&&r_0G_0(t)-a_0(t)t^2\geq\bar{c}t^p,\ pG_0(t)-a_0(t)t^2\geq-\bar{c}_0
\end{eqnarray*}
for all $t>0$ and some $\bar{c},\ \bar{c}_0>0$.
\begin{remark}
	Hypotheses $H(a)(i),(ii),(iii)$ are motivated by the nonlinear regularity theory of Lieberman \cite{23} and the nonlinear maximum principle of Pucci and Serrin \cite{40}. Hypothesis $H(a)(iv)$ serves the particular needs of our problem, but it is not restrictive and it is satisfied in many cases of interest as the examples below illustrate. Similar conditions were also used in the recent works of the authors, see Papageorgiou and R\u adulescu \cite{29, 35, 31}.
\end{remark}

Hypotheses $H(a)(i),(ii),(iii)$ imply that $G_0(\cdot)$ is strictly convex and strictly increasing. We set $G(y)=G_0(|y|)$ for all $y\in\RR^N$. So, $G(\cdot)$ is convex, $G(0)=0$ and
$$\nabla G(y)=G'_0(|y|)\frac{y}{|y|}=a_0(|y|)y\ \mbox{for all}\ y\in\RR^N\backslash\{0\},\ \nabla G(0)=0.$$

Therefore $G(\cdot)$ is the primitive of $a(\cdot)$. From the convexity of $G(\cdot)$ and since $G(0)=0$, we have
\begin{equation}\label{eq2}
	G(y)\leq (a(y),y)_{\RR^N}\ \mbox{for all}\ y\in\RR^N.
\end{equation}

The next lemma summarizes the main properties of the map $a(\cdot)$, which we will use in the sequel. These properties are straightforward consequences of properties $H(a)(i),(ii),(iii)$ and of (\ref{eq1}).
\begin{lemma}\label{lem2}
	If hypotheses $H(a)(i),(ii),(iii)$ hold, then
	\begin{itemize}
		\item[(a)] $y\mapsto a(y)$ is continuous and strictly monotone (hence maximal monotone, too);
		\item[(b)] $|a(y)|\leq c_4(1+|y|^{p-1})$ for all $y\in\RR^N$, some $c_4>0$;
		\item[(c)] $(a(y),y)_{\RR^N}\geq\frac{c_1}{p-1}|y|^p$ for all $y\in\RR^N$.
	\end{itemize}
\end{lemma}

This lemma and (\ref{eq2}) lead to the following growth estimates for the primitive $G(\cdot)$.
\begin{corollary}\label{cor3}
	If hypotheses $H(a)(i),(ii),(iii)$ hold, then $\frac{c_1}{p(p-1)}|y|^p\leq G(y)\leq c_5(1+|y|^p)$ for all $y\in\RR^N$, some $c_5>0$.
\end{corollary}

The examples which follow confirm the generality of hypotheses $H(a)$.
\begin{example}
	The following maps satisfy hypotheses $H(a)$ above
	\begin{itemize}
		\item[(a)] $a(y)=|y|^{p-2}y$ with $1<p<\infty$.

		The corresponding differential operator is the $p$-Laplacian defined by
		$$\Delta_pu={\rm div}\,(|Du|^{p-2}Du)\ \mbox{for all}\ u\in W^{1,p}(\Omega).$$
		\item[(b)] $a(y)=|y|^{p-2}y+|y|^{q-2}y$ with $1<q<p<\infty.$
		
		The corresponding differential operator is the $(p,q)$-Laplacian defined by
		$$\Delta_pu+\Delta_qu\ \mbox{for all}\ u\in W^{1,p}(\Omega).$$
		
		Such operators arise in problems of mathematical physics, see Benci, D'Avenia, Fortunato and Pisani \cite{6} (quantum physics) and Cherfils and Ilyasov \cite{8} (plasma physics). Recently there have been some existence and multiplicity results for such equations. We mention the papers of Aizicovici, Papageorgiou and Staicu \cite{1,2}, Cingolani and Degiovanni \cite{9}, Mugnai and Papageorgiou \cite{26}, Papageorgiou and R\u adu\-les\-cu \cite{28, 30, 32}, Papageorgiou, R\u adulescu and Repov\v{s} \cite{37}, Papageorgiou and Winkert \cite{39}, Sun, Zhang and Su \cite{41}.
		\item[(c)] $a(y)=(1+|y|^2)^{\frac{p-2}{2}}y$ with $1<p<\infty.$
		
		The corresponding differential operator is the generalized $p$-mean curvature differential operator defined by
		$${\rm div}\,\left((1+|Du|^2)^{\frac{p-2}{2}}Du\right)\ \mbox{for all}\ u\in W^{1,p}(\Omega).$$
		\item[(d)] $a(y)=|y|^{p-2}y+\frac{|y|^{p-2}y}{1+|y|^p}$ with $1<p<\infty$.
		
		The corresponding differential operator is defined by
		$$\Delta_pu+{\rm div}\,\left(\frac{|Du|^{p-2}Du}{1+|Du|^p}\right)\ \mbox{for all}\ u\in W^{1,p}(\Omega).$$
		
		This operator arises in problems of plasticity (see Fuchs and Osmolovski \cite{12}).
	\end{itemize}
\end{example}

Let $A:W^{1,p}(\Omega)\rightarrow W^{1,p}(\Omega)^*$ be the nonlinear map defined by
$$\left\langle A(u),h\right\rangle=\int_{\Omega}(a(Du),Dh)_{\RR^N}dz\ \mbox{for all}\ u,h\in W^{1,p}(\Omega).$$

The next proposition is a particular case of a more general result due to Gasinski and Papageorgiou \cite{15}.
\begin{prop}\label{prop4}
	If hypotheses $H(a)(i),(ii),(iii)$ hold, then the map $A:W^{1,p}(\Omega)\rightarrow W^{1,p}(\Omega)^*$ is continuous, monotone (hence maximal monotone too) and of type $(S)_+$, that is,
	$$``u_n\stackrel{w}{\rightarrow}u\ \mbox{in}\ W^{1,p}(\Omega)\ \mbox{and}\ \limsup\limits_{n\rightarrow\infty}\left\langle A(u_n),u_n-u\right\rangle\leq 0\Rightarrow u_n\rightarrow u\ \mbox{in}\ W^{1,p}(\Omega)."$$
\end{prop}

We introduce the following conditions on the coefficient functions $\xi(\cdot)$ and $\beta(\cdot)$.

\smallskip
$H(\xi):$ $\xi\in L^{\infty}(\Omega),\ \xi(z)\geq 0$ for almost all $z\in\Omega$.

\smallskip
$H(\beta):$ $\beta\in C^{0,\alpha}(\partial\Omega)$ with $\alpha\in(0,1),\ \beta(z)\geq 0$ for all $z\in\partial\Omega$.

\smallskip
$H_0:$ $\xi\neq 0$ or $\beta\neq 0$.
\begin{lemma}\label{lem5}
	If $\hat{\xi}\in L^{\infty}(\Omega),\hat{\xi}(z)\geq 0$ for almost all $z\in\Omega,\hat{\xi}\neq 0$, then there exists $c_6>0$ such that
	$$||Du||^p_p+\int_{\Omega}\hat{\xi}(z)|u|^pdz\geq c_6||u||^p\ \mbox{for all}\ u\in W^{1,p}(\Omega).$$
\end{lemma}
\begin{proof}
	Let $\psi:W^{1,p}(\Omega)\rightarrow\RR_+$ be the $C^1$-functional defined by
	$$\psi(u)=||Du||^p_p+\int_{\Omega}\hat{\xi}(z)|u|^pdz\ \mbox{for all}\ u\in W^{1,p}(\Omega).$$
	
	Arguing by contradiction, suppose that the lemma is not true. Since $\psi(\cdot)$ is $p$-homogeneous, we can find $\{u_n\}_{n\geq 1}\subseteq W^{1,p}(\Omega)$ such that
	\begin{equation}\label{eq3}
		||u_n||=1\ \mbox{for all}\ n\in\NN\ \mbox{and}\ \psi(u_n)\rightarrow 0^+\ \mbox{as}\ n\rightarrow\infty\,.
	\end{equation}
	
	Since $\{u_n\}_{n\geq 1}\subseteq W^{1,p}(\Omega)$ is bounded, we may assume that
	\begin{equation}\label{eq4}
		u_n\stackrel{w}{\rightarrow}u\ \mbox{in}\ W^{1,p}(\Omega)\ \mbox{and}\ u_n\rightarrow u\ \mbox{in}\ L^p(\Omega).
	\end{equation}
	
	The functional $\psi(\cdot)$ is sequentially weakly lower semicontinuous. So, from (\ref{eq3}) and (\ref{eq4}) we obtain
	\begin{eqnarray}\label{eq5}
		&&\psi(u)\leq 0,\nonumber\\
		&\Rightarrow&||Du||^p_p\leq-\int_{\Omega}\hat{\xi}(z)|u|^pdz\leq 0,\\
		&\Rightarrow&u=\eta\in\RR\,.\nonumber
	\end{eqnarray}
	
	If $\eta=0$, then from (\ref{eq4}) we see that
	\begin{eqnarray*}
		&&||Du_n||_p\rightarrow 0,\\
		&\Rightarrow&u_n\rightarrow 0\ \mbox{in}\ W^{1,p}(\Omega),
	\end{eqnarray*}
	a contradiction to the fact that $||u_n||=1$ for all $n\in\NN$.
	
	If $\eta\neq 0$, then from (\ref{eq5}) we have
	$$0\leq-|\eta|^p\int_{\Omega}\hat{\xi}(z)dz<0,$$
	a contradiction.
	
	This proves the lemma.
\end{proof}
\begin{lemma}\label{lem6}
	If $\hat{\beta}\in L^{\infty}(\partial\Omega),\hat{\beta}(z)\geq 0$ for $\sigma$-almost all $z\in\partial\Omega,\hat{\beta}\neq 0$, then there exists $c_7>0$ such that
	$$||Du||^p_p+\int_{\partial\Omega}\hat{\beta}(z)|u|^{p}d\sigma\geq c_7||u||^p\ \mbox{for all}\ u\in W^{1,p}(\Omega).$$
\end{lemma}
\begin{proof}
	Let $\psi_0:W^{1,p}(\Omega)\rightarrow\RR_+$ be the $C^1$-functional defined by
	$$\psi_0(u)=||Du||^p_p+\int_{\partial\Omega}\beta(z)|u|^pd\sigma\ \mbox{for all}\ u\in W^{1,p}(\Omega).$$
	
	We claim that we can find $\hat{c}_0>0$ such that
	\begin{equation}\label{eq6}
		||u||^p_p\leq\hat{c}_0\psi_0(u)\ \mbox{for all}\ u\in W^{1,p}(\Omega).
	\end{equation}
	
	Arguing by contradiction, suppose that (\ref{eq6}) is not true. Then we can find $\{u_n\}_{n\geq 1}\subseteq W^{1,p}(\Omega)$ such that
	\begin{equation}\label{eq7}
		||u_n||^p_p>n\psi_0(u_n)\ \mbox{for all}\ n\in\NN\,.
	\end{equation}
	
	Since $\psi_0$ is $p$-homogeneous, we normalize in $L^p(\Omega)$ and have
	\begin{eqnarray}\label{eq8}
		&&\psi_0(u_n)<\frac{1}{n}\ \mbox{and}\ ||u_n||_p=1\ \mbox{for all}\ n\in\NN\ (\mbox{see (\ref{eq7})})\nonumber\\
		&\Rightarrow&\psi_0(u_n)\rightarrow 0^+\ \mbox{as}\ n\rightarrow\infty\,.
	\end{eqnarray}
	
	From (\ref{eq8}) it follows that
	\begin{eqnarray*}
		&&||Du_n||_p\rightarrow 0\ \mbox{as}\ n\rightarrow\infty,\\
		&\Rightarrow&\{u_n\}_{n\geq 1}\subseteq W^{1,p}(\Omega)\ \mbox{is bounded}.
	\end{eqnarray*}
	
	So, by passing to a suitable subsequence if necessary, we may assume that
	\begin{equation}\label{eq9}
		u_n\stackrel{w}{\rightarrow}u\ \mbox{in}\ W^{1,p}(\Omega)\ \mbox{and}\ u_n\rightarrow u\ \mbox{in}\ L^p(\Omega)\ \mbox{and in}\ L^p(\partial\Omega).
	\end{equation}
	
	From (\ref{eq8}), (\ref{eq9}) and the sequential weak lower semicontinuity of $\psi_0(\cdot)$, we have
	\begin{eqnarray}\label{eq10}
		&&\psi_0(u)\leq 0,\nonumber\\
		&\Rightarrow&||Du||^p_p+\int_{\partial\Omega}\beta(z)|u|^pd\sigma\leq 0,\\
		&\Rightarrow&u=\eta_0\in\RR\,.\nonumber
	\end{eqnarray}
	
	If $\eta_0=0$, then from (\ref{eq9}) we have
	$$u_n\rightarrow 0\ \mbox{in}\ L^p(\Omega),$$
	a contradiction with the fact that $||u_n||_p=1$ for all $n\in\NN$.
	
	If $\eta_0\neq 0$, then from (\ref{eq10}) we have
	$$0<|\eta_0|^p\int_{\partial\Omega}\beta(z)d\sigma\leq 0,$$
	again a contradiction. Therefore (\ref{eq6}) holds and from this it follows that we can find $c_7>0$ such that
	$$c_7||u||^p\leq\psi_0(u)\ \mbox{for all}\ u\in W^{1,p}(\Omega).$$
\end{proof}

Next we prove a strong comparison result which will be useful in what follows. This proposition was inspired by analogous comparison results for Dirichlet problems with the $p$-Laplacian as established  by Guedda and V\'eron \cite[Proposition 2.2]{17}  and Arcoya and Ruiz \cite[Proposition 2.6]{5}.
\begin{prop}\label{prop7}
	Assume that hypotheses $H(a)(i),(ii),(iii)$ hold, $\hat{\xi}\in L^{\infty}(\Omega),\ \hat{\xi}(z)\geq 0$ for almost all $z\in\Omega$, $h_1,h_2\in L^{\infty}(\Omega)$ such that
	$$0<c_8\leq h_2(z)-h_1(z)\ \mbox{for almost all}\ z\in\Omega$$
	$u,v\in C^1(\overline{\Omega})\backslash\{0\}$ satisfy $u\leq v$ and
	\begin{eqnarray*}
		&&-{\rm div}\, a(Du(z))+\hat{\xi}(z)|u(z)|^{p-2}u(z)=h_1(z)\ \mbox{for almost all}\ z\in\Omega,\\
		&&-{\rm div}\, a(Dv(z))+\hat{\xi}(z)|v(z)|^{p-2}v(z)=h_2(z)\ \mbox{for almost all}\ z\in\Omega.
	\end{eqnarray*}
	Then $(v-u)(z)>0$ for all $z\in\Omega$ and if $\Sigma_0=\{z\in\partial\Omega:u(z)=v(z)\}$, then
	$$\left.\frac{\partial(v-u)}{\partial n}\right|_{\Sigma_0}<0.$$
\end{prop}
\begin{proof}
	We have
	\begin{eqnarray}\label{eq11}
		&&-{\rm div}\, (a(Dv(z))-a(Du(z)))\nonumber\\
		&&=h_2(z)-h_1(z)-\hat{\xi}(z)(|v(z)|^{p-2}v(z)-|u(z)|^{p-2}u(z))\ \mbox{for almost all}\ z\in\Omega\,.
	\end{eqnarray}
	
	Let $a=(a_k)^N_{k=1}$ with $a_k:\RR^N\rightarrow\RR$ being the $k$th component function, $k\in\{1,\ldots,N\}$. From the mean value theorem, we have
	$$a_k(y)-a_k(y')=\overset{N}{\underset{\mathrm{i=1}}\sum}\int^1_0\frac{\partial a_k}{\partial y_i}(y'+t(y-y'))(y_i-y'_i)dt$$
	for all $y=(y_i)^N_{i=1}\in\RR^N$, $y'=(y'_i)^N_{i=1}\in\RR^N$ and all $k\in\{1,\ldots,N\}$.
	
	Consider the following functions
	\begin{eqnarray*}
		&&\tilde{c}_{k,i}(z)=\int^1_0\frac{\partial a_k}{\partial y_i}(Du(z)+t(Dv(z)-Du(z)))(Dv(z)-Du(z))dt\\
		&&\mbox{for all}\ k\in\{1,\ldots,N\},\ \mbox{all}\ z\in\Omega\,.
	\end{eqnarray*}
	
	Then $\tilde{c}_{k,i}\in C(\overline{\Omega})$ and using these functions we introduce the following linear differential operator in  divergence form
	$$L(w)=-{\rm div}\,\left(\overset{N}{\underset{\mathrm{i=1}}\sum}\tilde{c}_{k,i}(z)\frac{\partial w}{\partial z_i}\right)=-\overset{N}{\underset{\mathrm{k,i=1}}\sum}\frac{\partial}{\partial z_k}(\tilde{c}_{k,i}(z)\frac{\partial w}{\partial z_i})\ \mbox{for all}\ w\in H^1(\Omega).$$
	
	We set $y=v-u\in C_+\backslash\{0\}$. From (\ref{eq11}) we have
	\begin{equation}\label{eq12}
		L(y)=h_2(z)-h_1(z)-\hat{\xi}(z)(|v(z)|^{p-2}v(z)-|u(z)|^{p-2}u(z))\ \mbox{for almost all}\ z\in\Omega\,.
	\end{equation}
	
	Suppose that at $z_0\in\Omega$, we have $u(z_0)=v(z_0)$. Exploiting the uniform continuity of the map $x\mapsto |x|^{p-2}x$ and the fact that $\hat{\xi}\in L^{\infty}(\Omega)$, from (\ref{eq12}) we see that for $\delta>0$ sufficiently small we have
	$$L(y)\geq\frac{c_8}{2}>0\ \mbox{for almost all}\ z\in B_{\delta}(z_0).$$
	
	Then invoking Harnack's inequality (see Motreanu, Motreanu and Papageorgiou \cite[p. 212]{25}) or alternatively using the tangency principle of Pucci and Serrin \cite[p. 35]{40}, we have
	$$(v-u)(z)>0\ \mbox{for all}\ z\in B_{\delta}(z_0),$$
	a contradiction since $u(z_0)=v(z_0)$. Therefore, we must have that
	$$(v-u)(z)>0\ \mbox{for all}\ z\in\Omega.$$
	
	Next suppose that $\hat{z}_0\in\Sigma_0$. Since $\partial\Omega$ is a $C^2$-manifold, for $\rho>0$ small there is a $\rho$-ball $B_{\rho}$ such that
	$$B_{\rho}\subseteq\Omega\ \mbox{and}\ \hat{z}_0\in\partial\Omega\cap\partial B_{\rho}.$$
	
	Choosing $\rho>0$ small, from (\ref{eq12}) and since $u(\hat{z}_0)=v(\hat{z}_0)$ (recall that $\hat{z}_0\in\Sigma_0$), we see that $L(\cdot)$ is strictly elliptic. Then Hopf's theorem (see Motreanu, Motreanu and Papageorgiou \cite[p. 217]{25}) and Pucci and Serrin \cite[p. 120]{40}, we have
	\begin{eqnarray*}
		&&\frac{\partial y}{\partial n}(z_0)=\frac{\partial(v-u)}{\partial n}(z_0)<0,\\
		&\Rightarrow&\left.\frac{\partial(v-u)}{\partial n}\right|_{\Sigma_0}<0.
	\end{eqnarray*}
\end{proof}

\begin{remark}
	With $\Sigma_0=\{z\in\partial\Omega:u(z)=v(z)\}$, we introduce the following Banach spaces:
	\begin{eqnarray*}
		&&C^1_*(\overline{\Omega})=\{h\in C^1(\overline{\Omega}):h|_{\Sigma_0}=0\},\\
		&&W^{1,p}_*(\Omega)=\overline{C^1_*(\overline{\Omega})}^{||\cdot||}\ (\mbox{recall that}\ ||\cdot||\ \mbox{is the norm of}\ W^{1,p}(\Omega)).
	\end{eqnarray*}
	
	From Proposition \ref{prop7} we have
	$$\left.\frac{\partial(v-u)}{\partial n}\right|_{\Sigma_0}\leq-\eta<0.$$
	
	Let $U$ be a neighborhood of $\Sigma_0$ in $\overline{\Omega}$ such that
	$$\left.\frac{\partial(v-u)}{\partial n}\right|_U\leq-\frac{\eta}{2}<0.$$
	
	Then we can find $\epsilon>0$ small such that
	\begin{eqnarray}
		h\in C^1_*(\overline{\Omega}),\ ||h||_{C^1(\overline{\Omega})}\leq\epsilon&\Rightarrow&\frac{\partial(v-(u+h))}{\partial h}\leq-\frac{\eta}{4}<0\label{eq13}\\
			&\mbox{and}&(v-(u_0+h))|_{\overline{\Omega}\backslash U}\geq\hat{\eta}>0.\label{eq14}
	\end{eqnarray}
	
	From (\ref{eq13}) we see that for $\epsilon>0$ small, we have
	$$v(z)-(u+h)(z)\geq 0\ \mbox{for all}\ z\in U,\ \mbox{all}\ h\in C^1_*(\overline{\Omega}),\ ||h||_{C^1(\overline{\Omega})}\leq\epsilon\,.$$
	
	Comparing this with (\ref{eq14}), we see that
	$$u+B^c_{\epsilon}\in v-C^*_+(\Sigma_0)$$
	with $B^c_{\epsilon}$ being the $\epsilon$-ball centered at zero in $C^1_*(\overline{\Omega})$ and $C^*_+(\Sigma_0)$ is the positive cone of $C^1_*(\overline{\Omega})$. This cone has a nonempty interior given by
	$${\rm int}\, C^*_+(\Sigma_0)=\{h\in C^*_+:h(z)>0\ \mbox{for all}\ z\in\Omega,\left.\frac{\partial h}{\partial n}\right|_{\Sigma_0}<0\}.$$
	
	If $\Sigma_0=\emptyset$, then $v-u\in D_+$.
\end{remark}

The next result is an outgrowth of the nonlinear regularity theory of Lieberman \cite{23} and can be found in Papageorgiou and R\u adulescu \cite{27} (subcritical case) and in Papageorgiou and R\u adulescu \cite{36} (critical case).

So, let $V$ and $X$ be two Banach subspaces of $C^1(\overline{\Omega})$ and $W^{1,p}(\Omega)$ respectively, such that $V$ is dense in $X$. Suppose that $f_0:\Omega\times\RR\rightarrow\RR$ is a Carath\'eodory function such that
$$|f_0(z,x)|\leq a_0(z)(1+|x|^{r-1})\ \mbox{for almost all}\ z\in\Omega,\ \mbox{all}\ x\in\RR,$$
with $a_0\in L^{\infty}(\Omega),1<r\leq p^*$. We set $F_0(z,x)=\int^x_0f_0(z,s)ds$ and consider the $C^1$-functional $\varphi_0:W^{1,p}(\Omega)\rightarrow\RR$ defined by
$$\varphi_0(u)=\int_{\Omega}G(Du)dz+\frac{1}{p}\int_{\partial\Omega}\beta(z)|u|^pd\sigma-\int_{\Omega}F_0(z,u)dz\ \mbox{for all}\ u\in W^{1,p}(\Omega).$$
\begin{prop}\label{prop8}
	Assume that $u_0\in W^{1,p}(\Omega)$ is a local $V$-minimizer of $\varphi_0$, that is, there exists $\rho_0>0$ such that
	$$\varphi_0(u_0)\leq\varphi_0(u_0+h)\ \mbox{for all}\ h\in V,\ ||h||_{C^1(\overline{\Omega})}\leq\rho_0.$$
	Then $u_0\in C^{1,\alpha}(\overline{\Omega})$ for some $\alpha\in(0,1)$ and $u_0$ is also a local $X$-minimizer of $\varphi_0$, that is, there exists $\rho_1>0$ such that
	$$\varphi_0(u_0)\leq\varphi_0(u_0+h)\ \mbox{for all}\ h\in X,\ ||h||\leq\rho_1.$$
\end{prop}

We conclude this section with some notation that we will use throughout this work. For every $x\in\RR$, let $x^{\pm}=\max\{\pm x,0\}$. Then for $u\in W^{1,p}(\Omega)$ we set $u^{\pm}(\cdot)=u(\cdot)^{\pm}$. We know that
$$u=u^+-u^-,|u|=u^++u^-\ \mbox{and}\ u^+,u^-\in W^{1,p}(\Omega).$$

By $|\cdot|_N$ we denote the Lebesgue measure on $\RR^N$. Finally, if $X$ is a Banach space and $\varphi\in C^1(X,\RR)$, then by $K_{\varphi}$ we denote the critical set of $\varphi$, that is,
$$K_{\varphi}=\{u\in X:\varphi'(u)=0\}.$$

\section{Bifurcation-Type Theorem}

In this section, we prove a bifurcation type theorem for problem \eqref{eqp} for small values of the parameter $\lambda>0$.

We introduce the following conditions on the reaction term $f(z,x)$.

\smallskip
$H(f):$ $f:\Omega\times\RR$ is a Carath\'eodory function such that for almost all $z\in\Omega$, $f(z,0)=0$, $f(z,x)>0$ for all $x>0$ and
\begin{itemize}
	\item[(i)] $f(z,x)\leq a(z)(1+x^{r-1})$ for almost all $z\in\Omega$, all $x\geq 0$, with $a\in L^{\infty}(\Omega)$, $p<r<p^*$;
	\item[(ii)] if $F(z,x)=\int^x_0f(z,s)ds$, then $\lim\limits_{x\rightarrow+\infty}\frac{F(z,x)}{x^p}=+\infty$ uniformly for almost all $z\in\Omega$;
	\item[(iii)] if $e(z,x)=f(z,x)x-pF(z,x)$, then there exists $d\in L^1(\Omega)$ such that
	$$e(z,x)\leq e(z,y)+d(z)\ \mbox{for almost all}\ z\in\Omega,\ \mbox{all}\ 0\leq x\leq y;$$
	\item[(iv)] for every $s>0$, we can find $\eta_s>0$ such that
	$$\eta_s\leq \inf[f(z,x):x\geq s]\ \mbox{for almost all}\ z\in\Omega$$
	and there exist $\delta_0>0,\ \hat{\eta},\ \hat{\eta}_0>0$ and $\tau\in(1,q)$ (see hypothesis $H(a)(iv)$) such that
	$$\hat{\eta}_0x^{\tau-1}\leq f(z,x)\leq\hat{\eta}x^{\tau-1}\ \mbox{for almost all}\ z\in\Omega,\ \mbox{all}\ 0\leq x\leq\delta_0;$$
	\item[(v)] for every $\rho>0$, there exists $\hat{\xi}_{\rho}>0$ such that for almost all $z\in\Omega$, the function
	$$x\mapsto f(z,x)+\hat{\xi}_{\rho}x^{p-1}$$
	is nondecreasing on $[0,\rho]$.
\end{itemize}

\begin{remark}
	Since we are looking for positive solutions and the above hypotheses concern the positive semiaxis, without any loss of generality, we may assume that $f(z,x)=0$ for almost all $z\in\Omega$, all $x\leq 0$. Hypotheses $H(f)(ii),(iii)$ imply that
	$$\lim\limits_{x\rightarrow+\infty}\frac{f(z,x)}{x^{p-1}}=+\infty\ \mbox{uniformly for almost all}\ z\in\Omega.$$
	
	So, the reaction term $f(z,\cdot)$ is $(p-1)$-superlinear. However, we stress that we do not use the usual for ``superlinear'' problems AR-condition. We recall that the AR-condition (unilateral version since we deal only with the positive semiaxis) says that there exist $\vartheta>p$ and $M>0$ such that
	\begin{equation}
		0<\vartheta F(z,x)\leq f(z,x)x\ \mbox{for almost all}\ z\in\Omega,\ \mbox{all}\ x\geq M,\tag{15a}\label{eq15a}
	\end{equation}
	\begin{equation}
		0<{\rm ess}\,\inf\limits_{\Omega}F(\cdot,M)\ (\mbox{see \cite{4}}).\tag{15b}\label{eq15b}
	\end{equation}
	\stepcounter{equation}
	
	Integrating \eqref{eq15a} and using \eqref{eq15b}, we obtain the weaker condition
	\begin{equation}\label{eq16}
		c_9x^{\vartheta}\leq F(z,x)\ \mbox{for almost all}\ z\in\Omega,\ \mbox{all}\ x\geq M,\ \mbox{some}\ c_9>0.
	\end{equation}
	
	Therefore the AR-condition implies that $f(z,\cdot)$ has at least $(\vartheta-1)$-polynomial growth near $+\infty$. This excludes from consideration $(p-1)$-superlinear nonlinearities with ``slower'' growth near $+\infty$ (see the examples below). For this reason in this work we use the less restrictive hypothesis $H(f)(iii)$. This is a quasimonotonicity condition on the function $e(z,\cdot)$. This is a slightly more general version of a condition used by Li and Yang \cite{24}. If there exists $M>0$ such that for almost all $z\in\Omega$ the function $x\mapsto\frac{f(z,x)}{x^{p-1}}$ is nondecreasing on $\left[M,+\infty\right)$, then hypothesis $H(f)(iii)$ is satisfied (see Li and Yang \cite{24}). Evidently this property is weaker than condition (\ref{eq16}).
\end{remark}	

\begin{example}
	The following functions satisfy hypotheses $H(f)$. For the sake of simplicity we drop the $z$-dependence.
	\begin{eqnarray*}
		&&f_1(x)=\left\{\begin{array}{ll}
			x^{\tau-1}&\mbox{if}\ x\in[0,1]\\
			x^{r-1}&\mbox{if}\ 1\leq x
		\end{array}\right.\ \mbox{with}\ 1<\,\tau<q<p<r<p^*\\
		&&f_2(x)=\left\{\begin{array}{ll}
			x^{\tau-1}-x^{s-1}&\mbox{if}\ x\in[0,1]\\
			x^{p-1}\ln x&\mbox{if}\ 1\leq x
		\end{array}\right.\ \mbox{with}\ 1<\tau<p,\,s.
	\end{eqnarray*}
	
	Note that $f_2(\cdot)$ does not satisfy the AR-condition.
\end{example}
	
	Hypotheses $H(f)(i),(iv)$ imply that
	\begin{equation}\label{eq17}
		0\leq f(z,x)\leq\hat{\eta}x^{\tau-1}+c_{10}x^{r-1}\ \mbox{for almost all}\ z\in\Omega,\ \mbox{all}\ x\geq 0,\ \mbox{some}\ c_{10}>0.
	\end{equation}
	
	This growth estimate on $f(z,\cdot)$ leads to the following auxiliary Robin problem:
	\begin{equation}
		\left\{\begin{array}{ll}
			-{\rm div}\,a(Du(z))+\xi(z)u(z)^{p-1}=\lambda(\hat{\eta}u(z)^{\tau-1}+c_{10}u(z)^{r-1})&\mbox{in}\ \Omega,\\
			\frac{\partial u}{\partial n_a}+\beta(z)u^{p-1}=0\ \mbox{on}\ \partial\Omega,\ u> 0,\ \lambda>0.
		\end{array}\right\}\tag{$Au_{\lambda}$}\label{eqa}
	\end{equation}

\begin{prop}\label{prop9}
	If hypotheses $H(a),H(\xi),H(\beta),H_0$ hold and $1<\tau<q<p<r<p^*$, then for $\lambda>0$ small problem \eqref{eqa} admits a positive solution $\tilde{u}_{\lambda}\in D_+$.
\end{prop}
\begin{proof}
	For $\lambda>0$, we consider the $C^1$-functional $\psi_{\lambda}:W^{1,p}(\Omega)\rightarrow\RR$ defined by
	\begin{eqnarray*}
		&&\psi_{\lambda}(u)=\int_{\Omega}G(Du)dz+\frac{1}{p}\int_{\Omega}\xi(z)|u|^pdz+\frac{1}{p}\int_{\partial\Omega}\beta(z)|u|^pd\sigma-\frac{\lambda\hat{\eta}}{\tau}||u^+||^{\tau}_{\tau}-\frac{\lambda c_{10}}{r}||u^+||^r_r\\
		&&\mbox{for all}\ u\in W^{1,p}(\Omega).
	\end{eqnarray*}
	\begin{claim}\label{cl1}
		For every $\lambda>0$ the functional $\psi_{\lambda}$ satisfies the C-condition.
	\end{claim}
	
	We consider a sequence $\{u_n\}_{n\geq 1}\subseteq W^{1,p}(\Omega)$ such that
	\begin{eqnarray}
		&&|\psi_{\lambda}(u_n)|\leq M_1\ \mbox{for some}\ M_1>0,\ \mbox{all}\ n\in\NN,\label{eq18}\\
		&&(1+||u_n||)\psi'_{\lambda}(u_n)\rightarrow 0\ \mbox{in}\ W^{1,p}(\Omega)^*\ \mbox{as}\ n\rightarrow\infty\,.\label{eq19}
	\end{eqnarray}
	
	From (\ref{eq19}) we have
	\begin{eqnarray}\label{eq20}
		&&\left|\left\langle A(u_n),h\right\rangle+\int_{\Omega}\xi(z)|u_n|^{p-2}u_nhdz+\int_{\partial\Omega}\beta(z)|u_n|^{p-2}u_nhd\sigma-\lambda\hat{\eta}\int_{\Omega}(u^+_n)^{\tau-1}hdz-\right.\nonumber\\
		&&\hspace{1cm}\left.-\lambda c_{10}\int_{\Omega}(u^+_n)^{r-1}hdz\right|\leq\frac{\epsilon_n||h||}{1+||u_n||}\\
		&&\mbox{for all}\ h\in W^{1,p}(\Omega)\ \mbox{as}\ n\rightarrow\infty\,.\nonumber
	\end{eqnarray}
	
	In (\ref{eq20}) we choose $h=-u^-_n\in W^{1,p}(\Omega)$. Then
	\begin{eqnarray}\label{eq21}
		&&\int_{\Omega}(a(-Du^-_n),-Du^-_n)_{\RR^N}dz+\int_{\Omega}\xi(z)(u^-_n)^pdz+\int_{\partial\Omega}\beta(z)(u^-_n)^pd\sigma\leq\epsilon_n\ \mbox{for all}\ n\in\NN,\nonumber\\
		&&\Rightarrow \frac{c_1}{p-1}||Du^-_n||^p_p+\int_{\Omega}\xi(z)(u^-_n)^pdz+\int_{\partial\Omega}\beta(z)(u^-_n)^pd\sigma\leq\epsilon_n\ \mbox{for all}\ n\in\NN\nonumber\\
		&&(\mbox{see Lemma \ref{lem2}}),\nonumber\\
		&&\Rightarrow c_{11}||u^-_n||^p\leq\epsilon_n\ \mbox{for some}\ c_{11}>0,\ \mbox{all}\ n\in\NN\ (\mbox{see hypotheses}\ H_0\ \mbox{and Lemmata \ref{lem5}, \ref{lem6}})\nonumber\\
		&&\Rightarrow u^-_n\rightarrow 0\ \mbox{in}\ W^{1,p}(\Omega).
	\end{eqnarray}
	
	We can always assume that $r_0\leq r<p^*$ (see hypotheses $H(a)(iv)$, $H(f)(i)$). From (\ref{eq18}) and (\ref{eq21}), we have that
	\begin{eqnarray}\label{eq22}
		&&\int_{\Omega}rG(Du^+_n)dz+\frac{r}{p}\int_{\Omega}\xi(z)(u^+_n)^p+\frac{r}{p}\int_{\partial\Omega}\beta(z)(u^+_n)^pd\sigma-\frac{\lambda\hat{\eta}r}{\tau}||u^+_n||^{\tau}_{\tau}-\lambda c_{10}||u^+_n||^r_r\leq M_2\\
		&&\mbox{for some}\ M_2>0,\ \mbox{all}\ n\in\NN\,.\nonumber
	\end{eqnarray}
	
	In (\ref{eq20}) we choose $h=u^+_n\in W^{1,p}(\Omega)$. Then
	\begin{eqnarray}\label{eq23}
		&&-\int_{\Omega}(a(Du^+_n),Du^+_n)_{\RR^N}dz-\int_{\Omega}\xi(z)(u^+_n)^pdz-\int_{\partial\Omega}\beta(z)(u^+_n)^pd\sigma+\lambda\hat{\eta}||u^+_n||^{\tau}_{\tau}+\nonumber\\
		&&\hspace{1cm}\lambda c_{10}||u^+_n||^{r}_{r}\leq\epsilon_n\ \mbox{for all}\ n\in\NN\,.
	\end{eqnarray}
	
	We add (\ref{eq22}) and (\ref{eq23}) and obtain
	\begin{eqnarray}\label{eq24}
		&&\int_{\Omega}\left[rG(Du^+_n)-(a(Du^+_n),Du^+_n)_{\RR^N}\right]dz+\left(\frac{r}{p}-1\right)\left[\int_{\Omega}\xi(z)(u^+_n)^pdz+\right.\nonumber\\
		&&\hspace{1cm}\left.\int_{\partial\Omega}\beta(z)(u^+_n)^pd\sigma\right]\ \leq M_3(1+\lambda||u^+_n||^{\tau}_{\tau})\ \mbox{for some}\ M_3>0,\ \mbox{all}\ n\in\NN\nonumber\\
		&\Rightarrow&c_{12}||u^+_n||^p\leq M_3(1+\lambda||u^+_n||^{\tau})\ \mbox{for some}\ c_{12}>0,\ \mbox{all}\ n\in\NN
	\end{eqnarray}
	(see hypotheses $H(a)(iv),H_0$, use Lemmata \ref{lem5}, \ref{lem6} and recall that $r>p$).
	
	Since $\tau<p$, from (\ref{eq24}) it follows that
	\begin{eqnarray*}
		&&\{u^+_n\}_{n\geq 1}\subseteq W^{1,p}(\Omega)\ \mbox{is bounded,}\\
		&\Rightarrow&\{u_n\}_{n\geq 1}\subseteq W^{1,p}(\Omega)\ \mbox{is bounded (see (\ref{eq21}))}.
	\end{eqnarray*}
	
	So, we may assume that
	\begin{eqnarray}\label{eq25}
		u_n\stackrel{w}{\rightarrow}u\ \mbox{in}\ W^{1,p}(\Omega)\ \mbox{and}\ u_n\rightarrow u\ \mbox{in}\ L^r(\Omega)\ \mbox{and in}\ L^p(\partial\Omega).
	\end{eqnarray}
	
	In (\ref{eq20}) we choose $h=u_n-u\in W^{1,p}(\Omega)$, pass to the limit as $n\rightarrow\infty$ and use (\ref{eq25}). Then
	\begin{eqnarray*}
		&&\lim\limits_{n\rightarrow\infty}\left\langle A(u_n),u_n-u\right\rangle=0,\\
		&\Rightarrow&u_n\rightarrow u\ \mbox{in}\ W^{1,p}(\Omega)\ (\mbox{see (\ref{eq25}) and Proposition \ref{prop4}}).
	\end{eqnarray*}
	
	Therefore for every $\lambda>0$, $\psi_{\lambda}$ satisfies the C-condition.
	
	This proves Claim \ref{cl1}.
	\begin{claim}\label{cl2}
		There exist $\rho>0$ and $\lambda_0>0$ such that for every $\lambda\in(0,\lambda_0)$ we have
		$$\inf[\psi_{\lambda}(u):||u||=\rho]=m_{\lambda}>0=\psi_{\lambda}(0).$$
	\end{claim}
		
		For every $u\in W^{1,p}(\Omega)$ we have
		\begin{eqnarray}\label{eq26}
			\psi_{\lambda}(u)&\geq&c_{13}||u||^p-\lambda c_{14}(||u||^{\tau}+||u||^r)\ \mbox{for some}\ c_{13},c_{14}>0\nonumber\\
			&&(\mbox{see Corollary \ref{cor3}, hypothesis $H_0$ and Lemmata \ref{lem5}, \ref{lem6}})\nonumber\\
			&&=[c_{13}-\lambda c_{14}(||u||^{\tau-p}+||u||^{r-p})]||u||^p.
		\end{eqnarray}
		
		Let $\Im(t)=t^{\tau-p}+t^{r-p},\ t>0$. Since $\tau<p<r$, we have
		$$\Im(t)\rightarrow+\infty\ \mbox{as}\ t\rightarrow 0^+\ \mbox{and as}\ t\rightarrow+\infty\,.$$
		
		Therefore we can find $t_0\in(0,+\infty)$ such that
		$$\Im(t_0)=\inf\limits_{t>0}\Im\,.$$
		
		From (\ref{eq26}) we see that
		\begin{equation}\label{eq27}
			\psi_{\lambda}(u)\geq[c_{13}-\lambda c_{14}\Im||u||]||u||^p\ \mbox{for all}\ u\in W^{1,p}(\Omega).
		\end{equation}
		
		If $||u||=t_0$, then we set $\lambda_0=\frac{c_{13}}{c_{14}\Im(t_0)}>0$ and for all $\lambda\in(0,\lambda_0)$ from (\ref{eq27}) we see that
		$$\inf[\psi_{\lambda}(u):||u||=\rho=t_0]=m_{\lambda}>0=\psi_{\lambda}(0).$$
		
		This proves Claim \ref{cl2}.
		
		Since $r>p$, if $u\in D_+$, then
		\begin{equation}\label{eq28}
			\psi_{\lambda}(tu)\rightarrow-\infty\ \mbox{as}\ t\rightarrow+\infty\,.
		\end{equation}
		
		Claims \ref{cl1} and \ref{cl2} and (\ref{eq28}) permit the use of Theorem \ref{th1} (the mountain pass theorem). So, for every $\lambda\in(0,\lambda_0)$, we can find $\tilde{u}_{\lambda}\in W^{1,p}(\Omega)$ such that
		\begin{equation}\label{eq29}
			\tilde{u}_{\lambda}\in K_{\psi_{\lambda}}\ \mbox{and}\ m_{\lambda}\leq\psi_{\lambda}(\tilde{u}_{\lambda}).
		\end{equation}
		
		From (\ref{eq29}) and Claim \ref{cl2} it follows that
		\begin{eqnarray}\label{eq30}
			&&\tilde{u}_{\lambda}\neq 0\ \mbox{and}\ \psi'_{\lambda}(\tilde{u}_{\lambda})=0,\nonumber\\
			&\Rightarrow&\left\langle A(\tilde{u}_{\lambda}),h\right\rangle+\int_{\Omega}\xi(z)|\tilde{u}_{\lambda}|^{p-2}\tilde{u}_{\lambda}hdz+\int_{\partial\Omega}\beta(z)|\tilde{u}_{\lambda}|^{p-2}\tilde{u}_{\lambda}hd\sigma\nonumber\\
			&&=\lambda \hat{\eta}\int_{\Omega}(\tilde{u}^+_{\lambda})^{\tau-1}hdz+\lambda c_{10}\int_{\Omega}(\tilde{u}^+_{\lambda})^{r-1}hdz\ \mbox{for all}\ h\in W^{1,p}(\Omega).
		\end{eqnarray}
		
		In (\ref{eq30}) we choose $h=-\tilde{u}^-_{\lambda}\in W^{1,p}(\Omega)$. Then
		\begin{eqnarray*}
			&&\frac{c_1}{p-1}||D\tilde{u}^-_{\lambda}||^p_p+\int_{\Omega}\xi(z)(\tilde{u}^-_{\lambda})^pdz+\int_{\partial\Omega}\beta(z)(\tilde{u}^-_{\lambda})^pd\sigma\leq 0\ (\mbox{see Lemma \ref{lem2}})\\
			&\Rightarrow&c_{15}||\tilde{u}^-_{\lambda}||^p\leq 0\ \mbox{for some}\ c_{15}>0\ (\mbox{see hypothesis $H_0$ and Lemmata \ref{lem5}, \ref{lem6}})\\
			&\Rightarrow&\tilde{u}_{\lambda}\geq 0,\ \tilde{u}_{\lambda}\neq 0.
		\end{eqnarray*}
		
		Then (\ref{eq30}) becomes
	\begin{eqnarray}\label{eq31}
		&&\left\langle A(\tilde{u}_\lambda), h \right\rangle + \int_\Omega \xi(z)\tilde{u}^{p-1}_\lambda hdz+\int_{\partial\Omega}\beta(z)\tilde{u}^{p-1}_\lambda hd\sigma=\int_\Omega \left[ \lambda \hat{\eta}\tilde{u}^{\tau-1}_\lambda + \lambda c_{10} \tilde{u}^{r-1}_\lambda \right]hdz\nonumber\\
		&&\mbox{for all}\ h\in W^{1,p}(\Omega),\nonumber\\
		&&\Rightarrow -{\rm div}\, a(D\tilde{u}_\lambda(z))+\xi(z)\tilde{u}_\lambda (z)^{p-1}=\lambda\left[ \hat{\eta}\tilde{u}_\lambda(z)^{\tau-1}+c_{10}\tilde{u}_\lambda (z)^{r-1} \right]\ \mbox{for almost all}\ z\in\Omega, \\
		&&\frac{\partial\tilde{u}_\lambda}{\partial n_a}+\beta (z) \tilde{u}^{p-1}_\lambda =0\ \mbox{on}\ \partial\Omega \nonumber
	\end{eqnarray}
	(see Papageorgiou and R\u adulescu \cite{27})
	
	From (\ref{eq31}) and Hu and Papageorgiou \cite{21} (subcritical case), Papageorgiou and R\u adulescu \cite{36} (critical case), we have
	$$\tilde{u}_\lambda\in L^\infty(\Omega).$$
	
	Then from Lieberman \cite{23} we infer that
	$$\tilde{u}_\lambda\in C_+ \backslash \{0\}.$$
	
	From (\ref{eq31}) we have
	\begin{eqnarray*}
		&	& {\rm div}\, a(D\tilde{u}_\lambda(z))\leq ||\xi||_\infty\tilde{u}_\lambda(z)^{p-1}\ \mbox{for almost all}\ z\in\Omega,\ \mbox{(see hypotheses H($\xi$)), H($\beta$))}\\
		&\Rightarrow & \tilde{u}_\lambda \in D_+\ \mbox{(see Pucci and Serrin \cite[pp. 111, 120]{40})}.
	\end{eqnarray*}
\end{proof}
	
	In fact we can show that for every $\lambda\in(0,\lambda_0)$, problem (\ref{eqa}) admits a smallest positive solution.

Let $\tilde{S}^\lambda_+$ be the set of positive solutions of problem (\ref{eqa}). We have seen in Proposition \ref{prop9} and its proof that
$$\emptyset\neq\tilde{S}^\lambda_+\subseteq D_+\ \mbox{for all}\ \lambda\in(0,\lambda_0).$$

Moreover, as in Filippakis and Papageorgiou \cite{11}, we have that $\tilde{S}^\lambda_+$ is downward directed (that is, if $\tilde{u}_1, \tilde{u}_2\in \tilde{S}^\lambda_+$, then we can find $\tilde{u}\in\tilde{S}^\lambda_+$, such that $\tilde{u}\leq\tilde{u}_1$ and $\tilde{u}\leq\tilde{u}_2$).
\begin{prop}\label{prop10}
	If hypotheses $H(a), H(\xi), H(\beta), H_0, H(f)$ hold and $\lambda\in(0,\lambda_0)$, then problem (\ref{eqa})	 admits a smallest positive solution $\tilde{u}_\lambda\in\tilde{S}^\lambda_+\subseteq D_+$ (that is, $\tilde{u}_\lambda\leq u$ for all $u\in\tilde{S}^\lambda_+$).
\end{prop}
\begin{proof}
	We consider the following Robin problem
	\begin{equation}\tag*{$(Au_\lambda)'$} \label{eqaulam'}
		\left\{\begin{array}{ll}
			-{\rm div}\, a(Du(z))+\xi(z)u(z)^{p-1}=\lambda\hat{\eta}u(z)^{\tau-1}\ \mbox{in}\ \Omega, \\
			\frac{\partial u}{\partial n_a}+\beta(z)u^{p-1}=0\ \mbox{on}\ \partial\Omega,\ u>0,\ \lambda>0.
		\end{array}\right\}
	\end{equation}
	
	Since $\tau<p$, a straightforward application of the direct method of the calculus of variations reveals that for every $\lambda>0$, problem \ref{eqaulam'} admits a positive solution $\overline{u}_\lambda\in D_+$ (nonlinear regularity theory and the nonlinear maximum principle).
	\begin{claim}\label{claim1}
		$\overline{u}_\lambda\in D_+$ is the unique positive solution of problem \ref{eqaulam'}.
	\end{claim}
	
	Consider the integral functional $j:L^1(\Omega)\rightarrow\overline{R}=\RR \cup\{+\infty\}$ defined by
	\begin{eqnarray*}
		j(u)=\left\{\begin{array}{cl}
		\int_\Omega G(Du^{^1/_q})dz+\frac{1}{p}\int_\Omega\xi(z)(u^{\frac{p}{q}})dz+\frac{1}{p}\int_{\partial\Omega}\beta(z)(u^{\frac{p}{q}})d\sigma\ & \mbox{if}\ u\geq0,\ w^{\frac{1}{q}}\in W^{1,p}(\Omega) \\
		+\infty	& \mbox{otherwise}.
	\end{array}\right.
	\end{eqnarray*}
	
	Let $u_1,u_2\in {\rm dom}\,j=\{u\in L^{1}(\Omega):j(u)<+\infty\}$ (the effective domain of the functional $j(\cdot)$) and set $u=((1-t)u_1+tu_2)^{^1/_q}$ with $t\in[0,1]$. Using Lemma 1 of Diaz and Saa \cite{10} we have
	\begin{equation}\label{eq32}
		|Du(z)|\leq\left[(1-t)|Du_1(z)^{\frac{1}{q}}|^q+t|Du_2(z)^{\frac{1}{q}}|^q\right]\ \mbox{for almost all}\ z\in\Omega	.
	\end{equation}
	
	Then we have
	\begin{eqnarray*}
		G_0(|Du(z)|)&\leq & G_0\left( (1-t)|Du_1(z)^{\frac{1}{q}}|^q+t|Du_2(z)^{\frac{1}{q}}|^q \right)\ \mbox{for almost all}\ z\in\Omega \\
					&	  & \mbox{(see (\ref{eq32}) and recall that $G_0(\cdot)$ is increasing)}\\
					&\leq & (1-t)G_0(|Du_1(z)^{\frac{1}{q}}|)+tG_0(|Du_2(z)|^{\frac{1}{q}})\ \mbox{for almost all}\ z\in\Omega \\
					&	  & \mbox{(see hypothesis H(a)(iv)),}\\
   \Rightarrow G(Du(z)) &\leq & (1-t)G(Du_1(z))^{\frac{1}{q}}+tG(Du_2(z)^{\frac{1}{q}})\ \mbox{for almost all}\ z\in\Omega, \\
   \Rightarrow j(\cdot)\ \mbox{is convex}& &\mbox{(recall that $q<p$ and see hypotheses H($\xi$), H($\beta$)).}
	\end{eqnarray*}
	
	By Fatou's lemma, we see that $j(\cdot)$ is also lower semicontinuous.
	
	Let $\overline{v}_\lambda\in W^{1,p}(\Omega)$ be another positive solution of problem \ref{eqaulam'}. Again we have $\overline{v}_\lambda\in D_+$. If $h\in C^1(\overline{\Omega})$, then for $t>0$ small we have
	$$\overline{u}^q_\lambda+th\in {\rm dom}\,j\ \mbox{and}\ \overline{v}^q_\lambda+th\in {\rm dom}\, j.$$
	
	Then we can easily show that $j(\cdot)$ is G\^ateaux differentiable at $\overline{u}^q_\lambda$ and at $\overline{v}^q_\lambda$ in the direction $h$. Moreover, via the chain rule and the nonlinear Green's theorem (see Gasinski and Papageorgiou \cite[p. 210]{14}), we have
	\begin{eqnarray*}
		j'(\overline{u}^q_\lambda)(h)&=&\frac{1}{q}\int_\Omega\frac{-{\rm div}\, a(D\overline{u}_\lambda)+\xi(z)\overline{u}^{p-1}_\lambda}{\overline{u}^{q-1}_\lambda}hdz \\
		j'(\overline{v}^q_\lambda)(h)&=&\frac{1}{q}\int_\Omega\frac{-{\rm div}\, a(D\overline{v}_\lambda)+\xi(z)\overline{v}^{p-1}_\lambda}{\overline{v}^{q-1}_\lambda}hdz \\
		&&\mbox{for all}\ h\in W^{1,p}(\Omega).
	\end{eqnarray*}
	The convexity of $j(\cdot)$ implies the monotonicity of $j'(\cdot)$. So
	\begin{eqnarray*}
		0 &\leq& \int_\Omega\left(\frac{-{\rm div}\, a(D\overline{u}_\lambda)+\xi(z)\overline{u}^{p-1}_\lambda}{\overline{u}^{q-1}_\lambda}-\frac{-{\rm div}\, a(D\overline{v}_\lambda)+\xi(z)\overline{v}^{p-1}_\lambda}{\overline{v}^{q-1}_\lambda}\right) (\overline{u}^q_\lambda-\overline{v}^q_\lambda)dz \\
		&\leq &\lambda\hat{\eta}\int_\Omega\left[\frac{1}{\overline{u}^{\tau-q}_\lambda}-\frac{1}{\overline{v}^{\tau-q}_\lambda}\right](\overline{u}^q_\lambda-\overline{v}^q_\lambda)dz\ \mbox{(see problem \ref{eqaulam'})}, \\
		\Rightarrow\overline{u}_\lambda &=& \overline{v}_\lambda\ \mbox{(since $\tau<q$).}
	\end{eqnarray*}
	
	This proves Claim \ref{claim1}.		
	\begin{claim}\label{claim2}
		$\overline{u}_\lambda\leq u$ for all $u\in \tilde{S}^\lambda_+$.
	\end{claim}
	
	Let $u\in \tilde{S}^\lambda_+$. We introduce the following Carath\'eodory function
	\begin{equation}\label{eq33}
		k_\lambda(z,x)=\left\{\begin{array}{lll}
			0	&	\mbox{if}\ x<0	& \\
			\lambda\hat{\eta}x^{\tau-1} & \mbox{if}\ 0\leq x\leq u(z) & \mbox{for all}\ (z,x)\in\Omega\times\RR	\\
			\lambda\hat{\eta}u(z)^{\tau-1} & \mbox{if}\ u(z)<x. 	
		\end{array}\right.
	\end{equation}
	
	We set $K_\lambda(z,x)=\int^x_0 k_\lambda(z,s)ds$ and consider the $C^{1}-$functional $\overline{\psi}_\lambda:W^{1,p}(\Omega)\rightarrow \RR$ defined by
	$$\overline{\psi}_\lambda(y)=\int_\Omega G(Dy)dz +\frac{1}{p}\int_\Omega\xi(z)|y|^pdz+\frac{1}{p}\int_{\partial\Omega}\beta(z)|y|^pd\sigma-\int_\Omega K_\lambda(z,y)dz\ \mbox{for all}\ y\in W^{1,p}(\Omega).$$
	
	From (\ref{eq33}), Lemma \ref{lem2} and hypothesis $H_0$ together with Lemmata \ref{lem5} and \ref{lem6}, we see that the functional $\overline{\psi}_\lambda$ is coercive. Also, the Sobolev embedding theorem and the compactness of the trace map, imply that $\overline{\psi}_\lambda$ is sequentially weakly lower semicontinuous. So, by the Weierstrass theorem, we can find $\overline{u}^*_\lambda\in W^{1,p}(\Omega)$ such that
	\begin{equation}\label{eq34}
		\overline{\psi}_\lambda(\overline{u}^*_\lambda)=\inf\left[\overline{\psi}_\lambda(u):u\in W^{1,p}(\Omega)\right].
	\end{equation}
	Hypothesis H(a)(iv) and Corollary \ref{cor3} imply that
	\begin{equation}\label{eq35}
		G(y)\leq c_{16}(|y|^q+|y|^p)\ \mbox{for all}\ y\in\RR^{N},\ \mbox{some}\ c_{16}>0.
	\end{equation}
	
	Since $\tau<q<p$, if $v\in D_+$, then for $t\in[0,1]$ small (such that $tv\leq u$, recall that $u\in D_+$), we have
	\begin{eqnarray*}
		&&\overline{\psi}(tv)\leq c_{16}t^q(||Dv||^q_q+||Dv||^p_p)+\frac{t^p}{p}\left[\int_\Omega\xi(z)v^pdz+\int_{\partial\Omega}\beta(z)v^pd\sigma\right]\\
		& & -\frac{\lambda\hat{\eta}t^\tau}{\tau}||v||^\tau_\tau < 0\ \mbox{(see (\ref{eq35}))}\\
		&\Rightarrow&\overline{\psi}_\lambda(\overline{u}^*_\lambda)<0=\overline{\psi}_\lambda(0)\ \mbox{(see (\ref{eq34}))},\\
		&\Rightarrow&\overline{u}^*_\lambda\neq0.
	\end{eqnarray*}
	
	From (\ref{eq34}) we have
	\begin{eqnarray}\label{eq36}
		&&\overline{\psi}'_\lambda(\overline{u}^*_\lambda)=0,\nonumber\\
		&\Rightarrow&\left\langle A(\overline{u}^*_\lambda),h\right\rangle+\int_\Omega\xi(z)|\overline{u}^*_\lambda|^{p-2}\overline{u}^*_\lambda hdz+\int_{\partial\Omega}\beta(z)|\overline{u}^*_\lambda|^{p-2}\overline{u}^*_\lambda hd\sigma=\int_\Omega k_\lambda(z,\overline{u}^*_\lambda)hdz
	\end{eqnarray}
	for all $h\in W^{1,p}(\Omega)$.
	
	In (\ref{eq36})  we first choose $-(\overline{u}^{*}_\lambda)^- \in W^{1,p}(\Omega)$. Then
	\begin{eqnarray*}
		&&c_{17}||(\overline{u}^*_\lambda)^-||^p\leq 0\ \mbox{for some}\ c_{17}>0\\
		&&\mbox{(see (\ref{eq34}), Lemma \ref{lem2}, hypothesis $H_0$ and Lemmata \ref{lem5}, \ref{lem6})}\\
		&\Rightarrow &\overline{u}^*_\lambda \geq0,\overline{u}^*_\lambda\neq0.
	\end{eqnarray*}
	
	Next, in (\ref{eq36}) we choose $h=(\overline{u}^*_\lambda-u)^+\in W^{1,p}(\Omega)$. Then
	\begin{eqnarray*}
		&&\left\langle A(\overline{u}^*_{\lambda}),(\overline{u}^*_\lambda -u)^+ \right\rangle +\int_\Omega\xi(z)(\overline{u}^*_{\lambda})^{p-1}(\overline{u}^*_{\lambda}-u)^+dz+\int_{\partial\Omega}\beta(z)(\overline{u}^*_{\lambda})^{p-1}(\overline{u}^*_{\lambda}-u)^+d\sigma \\
		&=&\int_\Omega\lambda\hat{\eta}u^{\tau-1}(\overline{u}^*_{\lambda}-u)^+dz\ \mbox{(see (\ref{eq33}))} \\
		&\leq &\int_\Omega\left[\lambda\hat{\eta}u^{\tau-1}+\lambda c_{10}u^{r-1}\right](\overline{u}^*_{\lambda}-u)^+dz \\
		&=&\left\langle A(u),(\overline{u}^*_{\lambda} -u)^+ \right\rangle + \int_\Omega\xi(z)u^{p-1}(\overline{u}^*_\lambda-u)^+dz+\int_{\partial\Omega}\beta(z)u^{p-1}(\overline{u}^*_\lambda-u)^+d\sigma\\
		&&\mbox{(since $u\in \tilde{S}^\lambda_+$)},\\
		&\Rightarrow &\overline{u}^*_\lambda\leq u.
	\end{eqnarray*}
	
	So, we have proved that
	\begin{eqnarray*}
		&&\overline{u}^*_\lambda\in[0,u]=\{y\in W^{1,p}(\Omega):0\leq y(z)\leq u(z)\ \mbox{for almost all}\ z\in\Omega \},\overline{u}^*_\lambda\neq 0,\\
		&\Rightarrow &\overline{u}^*_\lambda\ \mbox{is a positive solution of \ref{eqaulam'}} \\
		&\Rightarrow &\overline{u}^*_\lambda=\overline{u}_\lambda\ \mbox{(see Claim \ref{claim1})} \\
		&\Rightarrow &\overline{u}_\lambda \leq u\ \mbox{for all}\ u\in \tilde{S}^\lambda_+.
	\end{eqnarray*}
	
	This proves Claim \ref{claim2}.
	
	Invoking Lemma 3.10 of Hu and Papageorgiou \cite[p. 178]{19}, we can find a decreasing sequence $\{u_n\}_{n\geq 1}\subseteq \tilde{S}^\lambda_+$ such that
	$$\inf \tilde{S}^\lambda_+=\inf\limits_{n\geq 1}u_n.$$
	
	Evidently $\{u_n\}_{n\geq 1}\subseteq W^{1,p}(\Omega)$ is bounded and so we may assume that
	\begin{equation}\label{eq37}
		u_n\stackrel{w}{\rightarrow}\tilde{u}^*_\lambda\ \mbox{in}\ W^{1,p}(\Omega)\ \mbox{and}\ u_n\rightarrow\tilde{u}^*_\lambda\ \mbox{in}\ L^r(\Omega)\ \mbox{and in}\ L^p(\partial\Omega).
	\end{equation}
	
	In (\ref{eq36}) we choose $h=u_n-\tilde{u}^*_\lambda\in W^{1,p}(\Omega)$, pass to the limit as $n\rightarrow \infty$ and use (\ref{eq37}).
	
	Then
	\begin{eqnarray}\label{eq38}
		&&\lim\limits_{n\rightarrow\infty}\left\langle A(u_n),u_n-\tilde{u}^*_\lambda\right\rangle=0, \nonumber \\
		&\Rightarrow &u_n\rightarrow\tilde{u}^*_\lambda\ \mbox{in}\ W^{1,p}(\Omega)\ \mbox{(see (\ref{eq37}) and Proposition \ref{prop4}).}
	\end{eqnarray}
	So, if in (\ref{eq36}) we pass to the limit as $n\rightarrow\infty$ and use (\ref{eq38}), then
	\begin{equation}\label{eq39}
		\left\langle A(\tilde{u}^*_\lambda),h\right\rangle+\int_\Omega\xi(z)(\tilde{u}^*_\lambda)^{p-1}hdz+\int_{\partial\Omega}\beta(z)(\tilde{u}^*_\lambda)^{p-1} hd\sigma=\int_\Omega \left[\lambda\hat{\eta}(\tilde{u}^*_\lambda)^{\tau-1}+\lambda c_{10}(\tilde{u}^*_\lambda)^{r-1}\right]hdz\
	\end{equation}
	for all $h\in W^{1,p}(\Omega)$.
	
	Also, from Claim \ref{claim2} we have
	\begin{eqnarray}\label{eq40}
		&&\overline{u}_\lambda\leq u_n\ \mbox{for all}\ n\in\NN, \nonumber \\
		&\Rightarrow &\overline{u}_\lambda\leq\tilde{u}^*_\lambda.
	\end{eqnarray}
	
	From (\ref{eq39}) and (\ref{eq40}) it follows that
	$$\tilde{u}^*_\lambda\in \tilde{S}^\lambda_+\ \mbox{and}\ \tilde{u}^*_\lambda=\inf\tilde{S}^\lambda_+.$$	
\end{proof}

Let
$$\mathcal{L}=\{\lambda>0:\ \mbox{problem (\ref{eqp}) admits a positive solution}\}.$$
\begin{prop}\label{prop11}
	If hypotheses H(a), H($\xi$), H($\beta$), $H_0$, H(f) hold then $\mathcal{L}\neq\emptyset$.	
\end{prop}
\begin{proof}
	Let $\tilde{u}^*_\lambda\in \tilde{S}^\lambda_+\subseteq D_+$ be the minimal positive solution of problem (\ref{eqa}) $(\lambda\in(0,\lambda_0))$, see Proposition \ref{prop10}.
	
	We introduce the following truncation of the reaction term in problem (\ref{eqp})
	\begin{equation}\label{eq41}
		\gamma_\lambda(z,x)=\left\{\begin{array}{ll}
			\lambda f(z,x) & \mbox{if}\ x\leq\tilde{u}^*_\lambda(z)\\
			\lambda f(z,\tilde{u}^*_\lambda(z)) & \mbox{if}\ \tilde{u}^*_\lambda(z)<x.
		\end{array}\right.
	\end{equation}
	
	This is a Carath\'eodory function. We set $\Gamma_\lambda(z,x)=\int^x_0\gamma_\lambda(z,s)ds$ and consider the $C^1-$functional $\hat{\varphi}:W^{1,p}(\Omega)\rightarrow\RR$ defined by
	$$\hat{\varphi}_\lambda(u)=\int_\Omega G(Du)dz+\frac{1}{p}\int_\Omega\xi(z)|u|^pdz+\frac{1}{p}\int_{\partial\Omega}\beta(z)|u|^pd\sigma-\int_\Omega\Gamma_\lambda(z,u)dz\ \mbox{for all}\ u\in W^{1,p}(\Omega).$$
	
	From (\ref{eq41}), Corollary \ref{cor3}, hypothesis $H_0$ and Lemmata \ref{lem5}, \ref{lem6}, we see that $\hat{\varphi}_\lambda(\cdot)$ is coercive. Also, it is sequentially weakly lower semicontinuous. So, we can find $u_\lambda\in W^{1,p}(\Omega)$ such that
	\begin{equation}\label{eq42}
		\hat{\varphi}_\lambda(u_\lambda)=\inf\left[\hat{\varphi}_\lambda(u):u\in W^{1,p}(\Omega)\right].
	\end{equation}
	
	Let $\delta_0>0$ be as postulated by hypothesis H(f)(iv). Given $u\in D_+$, we can find $t\in(0,1)$ small such that
	\begin{equation}\label{eq43}
		tu(z)\in(0,\delta_0]\ \mbox{for all}\ z\in\overline{\Omega}.
	\end{equation}
	
	Then hypothesis H(f)(iv) implies that
	\begin{equation}\label{eq44}
		F(z,tu(z))\geq\frac{\hat{\eta}_0}{\tau}(tu(z))^\tau\ \mbox{for almost all}\ z\in\Omega\ \mbox{(see (\ref{eq43})).}
	\end{equation}
	
	We have
	\begin{eqnarray}\label{eq45}
		\hat{\varphi}_\lambda(tu)\leq c_{16}t^q(||Du||^q_q+||Du||^p_p)+\frac{t^p}{p}\int_\Omega\xi(z)u^pdz & + & \frac{t^p}{p}\int_{\partial\Omega}\beta(z)u^pd\sigma \nonumber \\
		& - & \frac{\lambda\hat{\eta}_0}{\tau}t^\tau||u||^\tau_\tau\ \nonumber \\
		&	& \mbox{(see (\ref{eq35}) and recall that $t\in(0,1)$)} \nonumber \\
		\leq c_{18}t^q-\lambda c_{19}t^p\ \mbox{for some}\ c_{18},\, c_{19}>0.
	\end{eqnarray}
	
	Since $\tau<q<p$, from (\ref{eq45}) it follows that by choosing $t\in(0,1)$ even smaller if necessary, we have
	\begin{eqnarray*}
		&&\hat{\varphi}(tu)<0,\\
		&\Rightarrow &\hat{\varphi}_\lambda(u_\lambda)<0=\hat{\varphi}_\lambda(0)\ \mbox{(see (\ref{eq42}))}\\
		&\Rightarrow &u_\lambda\neq0.
	\end{eqnarray*}
	
	From (\ref{eq42}) we have
	$$\hat{\varphi}'_\lambda(u_\lambda)=0,$$
	\begin{equation}\label{eq46}
		\Rightarrow\left\langle A(u_\lambda),h\right\rangle+\int_\Omega\xi(z)|u_\lambda|^{p-2}u_\lambda hdz+\int_{\partial\Omega}\beta(z)|u_\lambda|^{p-2}u_\lambda hd\sigma =\int_\Omega \gamma_\lambda(z,u_\lambda)hdz
	\end{equation}
	for all $h\in W^{1,p}(\Omega)$.
	
	In (\ref{eq46}) we choose $h=-u^-_\lambda\in W^{1,p}(\Omega)$. Then as before
	\begin{eqnarray*}
		&&c_{20}||u^-_\lambda||^p\leq0\ \mbox{for some}\ c_{20}>0, \\
		&\Rightarrow & u_\lambda \geq 0, u_\lambda\neq 0.
	\end{eqnarray*}
	
	Also, in (\ref{eq46}) we choose $h=(u_\lambda-\tilde{u}^*_\lambda)^+\in W^{1,p}(\Omega)$. Then
	\begin{eqnarray*}
		& &\left\langle A(u_\lambda),(u_\lambda-\tilde{u}^*_\lambda)^+\right\rangle+\int_\Omega\xi(z)u^{p-1}_\lambda (u_\lambda-\tilde{u}^*_\lambda)^+dz+\int_{\partial\Omega}\beta(z)u^{p-1}_\lambda(u_\lambda-\tilde{u}^*_\lambda)^+d\sigma \\
		&=&\int_\Omega\lambda f(z,\tilde{u}^*_\lambda)(u_\lambda-\tilde{u}^*_\lambda)^+dz\ \mbox{(see (\ref{eq41}))} \\
		&\leq & \int_\Omega \lambda\left[ \hat{\eta}(\tilde{u}^*_\lambda)^{\tau-1}+c_{10}(\tilde{u}^*_\lambda)^{r-1}\right](u_\lambda-\tilde{u}^*_\lambda)^+dz\ \mbox{(see (\ref{eq17}))} \\
		&=&\left\langle A(\tilde{u}^*_\lambda),(u_\lambda-\tilde{u}^*_\lambda)^+\right\rangle+\int_\Omega\xi(z)(\tilde{u}^*_\lambda)^{p-1} (u_\lambda-\tilde{u}^*_\lambda)^+dz+\\
		&&\int_{\partial\Omega}\beta(z)(\tilde{u}^*_\lambda)^{p-1}(u_\lambda-\tilde{u}^*_\lambda)^+d\sigma\ \mbox{(because $\tilde{u}^*_\lambda\in \tilde{S}^\lambda_+$),} \\
		&\Rightarrow & u_\lambda \leq\tilde{u}^*_\lambda.
	\end{eqnarray*}
	
	So, we have proved that
	\begin{eqnarray*}
		& &u_\lambda\in [0,\tilde{u}^*_\lambda],\ u_\lambda\neq0, \\	
		&\Rightarrow & u_\lambda\ \mbox{is a positive solution of problem (\ref{eqp}) (see (\ref{eq41})).}
	\end{eqnarray*}
	
	As before the nonlinear regularity theory implies that
	$$u_\lambda\in C_+\backslash\{0\}.$$
	
	Let $\rho=||u_\lambda||_\infty$ and let $\hat{\xi}_\rho>0$ be as postulated by hypothesis $H(f)(v)$. Then
	\begin{eqnarray*}
		& &-{\rm div}\,a(Du_\lambda(z))+(\xi(z)+\hat{\xi}_\rho)u_\lambda(z)^{p-1}\geq0\ \mbox{for almost all}\ z\in\Omega \\
		&\Rightarrow & {\rm div}\,a(Du_\lambda(z))\leq\left[||\xi||_\infty+\hat{\xi}_\rho\right]u_\lambda(z)^{p-1}\ \mbox{for almost all}\ z\in\Omega\ \mbox{(see hypothesis H($\xi$)),} \\
		&\Rightarrow & u_\lambda\in D_+\ \mbox{(see Pucci and Serrin \cite[pp. 111, 120]{40}).}
	\end{eqnarray*}
	
	Therefore we infer that
	$$(0,\lambda_0)\subseteq\mathcal{L},\ \mbox{hence}\ \mathcal{L}\neq\emptyset.$$
\end{proof}

Let $S^\lambda_+$ be the set of positive solutions of problem (\ref{eqp}). A byproduct of the proof of Proposition \ref{prop11} is the following corollary.
\begin{corollary}\label{cor12}
	If hypotheses $H(a), H(\xi), H(\beta), H_0, H(f)$ hold, then $S^\lambda_+\subseteq D_+$.
\end{corollary}

The next proposition reveals a basic property of the set $\mathcal{L}$ of admissible parameter values.
\begin{prop}\label{prop13}
	If hypotheses $H(a),H(\xi),H(\beta),H_0,H(f)$ hold, $\lambda\in\mathcal{L}$ and $\alpha\in(0,\lambda)$, then $\alpha\in\mathcal{L}$.
\end{prop}
\begin{proof}
	Since $\lambda\in\mathcal{L}$, we can find $u_\lambda\in S^\lambda_+\subseteq D_+$ (see Corollary \ref{cor12}). We introduce the Carath\'eodory function $\mu_\alpha:\Omega\times\RR\rightarrow\RR$ defined by
	\begin{equation}\label{eq47}
		\mu_\alpha(z,x)=\left\{\begin{array}{ll}
			\alpha f(z,x) & \mbox{if}\ x\leq u_\lambda(z)\\
			\alpha f(z,u_\lambda(z)) & \mbox{if}\ u_\lambda(z)<x.
		\end{array}\right.
	\end{equation}
	
	We set $M_\alpha(z,x)=\int^x_0\mu_\alpha(z,s)ds$ and consider the $C^1-$functional $w_\alpha:W^{1,p}(\Omega)\rightarrow\RR$ defined by
	$$w_\alpha(u)=\int_\Omega G(Du)dz+\frac{1}{p}\int_\Omega\xi(z)|u|^pdz+\frac{1}{p}\int_{\partial\Omega}\beta(z)|u|^pd\sigma-\int_\Omega M_\alpha(z,u)dz$$
	for all $u\in W^{1,p}(\Omega)$.
	
	Clearly, $w_\alpha(\cdot)$ is coercive (see (\ref{eq47})) and sequentially weakly lower semicontinuous. So, we can find $u_\alpha\in W^{1,p}(\Omega)$ such that
	\begin{equation}\label{eq48}
		w_\alpha(u_\alpha)=\inf[w_\alpha(u):u\in W^{1,p}(\Omega)].
	\end{equation}
	
	As before (see the proof of Proposition \ref{prop10}), using hypothesis $H(f)(iv)$, we have
	\begin{eqnarray*}
		&&w_\alpha(u_\alpha)<0=w_\alpha(0), \\
		&\Rightarrow & u_\alpha\neq 0.
	\end{eqnarray*}
	
	From (\ref{eq48}), we have
	$$w'_\alpha(u_\alpha)=0,$$
	\begin{equation}\label{eq49}
		\left\langle A(u_\alpha),h\right\rangle+\int_\Omega\xi(z)|u_\alpha|^{p-2} u_\alpha hdz + \int_{\partial\Omega}\beta(z)|u_\alpha|^{p-2}u_\alpha hd\sigma = \int_\Omega\mu_\alpha(z,u_\alpha)hdz
	\end{equation}
	for all $h\in W^{1,p}(\Omega)$.
	
	In (\ref{eq49}) we first choose $h=-u^-_\alpha\in W^{1,p}(\Omega)$. We obtain
	$$0\leq u_\alpha,\ u_\alpha\neq 0.$$
	
	Then we choose $h=(u_\alpha-u_\lambda)^+\in W^{1,p}(\Omega)$. We have
	\begin{eqnarray*}
		&&\left\langle A(u_\alpha),(u_\alpha-u_\lambda)^+\right\rangle+\int_\Omega\xi(z)u^{p-1}_\alpha(u_\alpha-u_\lambda)^+dz+\int_{\partial\Omega}\beta(z)u^{p-1}_\alpha(u_\alpha-u_\lambda)^+d\sigma \\
		&=&\int_\Omega\alpha f(z,u_\lambda)(u_\alpha-u_\lambda)^+dz\ \mbox{(see (\ref{eq47}))} \\
		&\leq & \int_\Omega\lambda f(z,u_\lambda)(u_\alpha-u_\lambda)^+dz\ \mbox{(since}\ f\geq0,\alpha\leq\lambda) \\
		&=&\langle A(u_\lambda),(u_\alpha-u_\lambda)^+\rangle +\int_{\Omega}\xi(z)u^{p-1}_\lambda(u_\alpha-u_\lambda)^+dz+\int_{\partial\Omega}\beta(z)u^{p-1}_\lambda(u_\alpha-u_\lambda)^+d\sigma\\
		&&\mbox{(since}\ u_\lambda\in S^\lambda_+)\\
		&\Rightarrow & u_\alpha\leq u_\lambda.
	\end{eqnarray*}
	
	So, we have proved that
	\begin{eqnarray*}
		&& u_\alpha\in[0,u_\lambda],\ u_\alpha\neq0, \\
		&\Rightarrow & u_\alpha\in S^\alpha_+ \subseteq D_+\ \mbox{(see (\ref{eq47})) and so}\ \alpha\in\mathcal{L}.
	\end{eqnarray*}
\end{proof}
\begin{remark}
	Proposition \ref{prop13} implies that $\mathcal{L}$ is an interval	
\end{remark}
\begin{corollary}\label{cor14}
	If hypotheses $H(a),\ H(\xi),\ H(\beta),\ H_0,\ H(f)$ hold, $\lambda\in\mathcal{L},\ \alpha\in(0,\lambda)$ and $u_\lambda\in S^\lambda_+\subseteq D_+$, then we can find $u_\alpha\in S^\alpha_+$ such that
	$$u_\lambda-u_\alpha\in {\rm int}\, C^*_+\left(\Sigma_0\right)$$
	with $\Sigma_0=\{z\in\partial\Omega:u_\lambda (z)=u_\alpha(z)\}$.
\end{corollary}
\begin{proof}
	From the proof of Proposition \ref{prop13}, we know that we can find $u_\alpha\in S^{\alpha}_+$ such that
	$$u_\lambda-u_\alpha\in C_+\backslash\{0\}.$$
	
	Let $\rho=||u_\lambda||_\infty$ and let $\hat{\xi}_\rho>0$ be as postulated by hypothesis $H(f)(v)$. Then we have
	\begin{eqnarray*}
		&&-{\rm div}\,a(Du_\alpha)+(\xi(z)+\alpha\hat{\xi}_\rho)u^{p-1}_\alpha \\
		&=&\alpha f(z,u_\alpha)+\alpha\hat{\xi}_\rho u^{p-1}_\alpha \\
		&\leq &\alpha f(z,u_\lambda)+\alpha\hat{\xi}_\rho u^{p-1}_\lambda\ \mbox{(see hypothesis H(f)(v) and recall that}\ u_\alpha\leq u_\lambda) \\
		&=&\lambda f(z,u_\lambda)-(\lambda-\alpha)f(z,u_\lambda)+\alpha\hat{\xi}_\rho u^{p-1}_\lambda \\
		&\leq &\lambda f(z,u_\lambda)-(\lambda-\alpha)\eta_s+\alpha\hat{\xi}_\rho u^{p-1}_\lambda\ \mbox{with}\ s=\min\limits_{\overline{\Omega}}u_\lambda>0\\
		&&\mbox{(see hypothesis H(f)(iv) and recall that\ $u_\lambda\in D_+$)} \\
		&<&-{\rm div}\,a(Du_\lambda)+\alpha\hat{\xi}_\rho u^{p-1}_\lambda\ \mbox{for almost all}\ z\in\Omega, \\
		&\Rightarrow & u_\lambda-u_\alpha\in {\rm int}\, C^*_+\left(\Sigma_0\right)\ \mbox{with}\ \Sigma_0=\{z\in\partial\Omega:u_\lambda(z)=u_\alpha(z)\}\ \mbox{(see Proposition \ref{prop7})}.
	\end{eqnarray*}
The proof is now complete.
\end{proof}

Now let $\lambda^*=\sup\mathcal{L}$.
\begin{prop}\label{prop15}
	If hypotheses $H(a),H(\xi),H(\beta),H_0,H(f)$ hold, then $\lambda^*<+\infty$.
\end{prop}
\begin{proof}
	Hypotheses $H(f)(i),(iv)$ and $H(\xi)$ imply that we can find $\overline{\lambda}>0$ big such that
	\begin{equation}\label{eq50}
		\overline{\lambda}f(z,x)-\xi(z)x^{p-1}\geq x^{p-1}\ \mbox{for almost all}\ z\in\Omega, \mbox{all}\ x\geq0.
	\end{equation}
	
	Let $\lambda>\overline{\lambda}$ and suppose that $\lambda\in\mathcal{L}$. Then we can find $u_\lambda\in S^\lambda_+\subseteq D_+$. So, we have
	$$m_\lambda=\min\limits_{\overline{\Omega}}u_\lambda>0.$$
	
	For $\delta>0$, we set $m^\delta_\lambda=m_\lambda+\delta\in D_+$. Also for $\rho=||u_\lambda||_\infty$ let $\hat{\xi}_{\rho}>0$ be as postulated by hypothesis $H(f)(v)$. Then
\begin{eqnarray}\label{eq51}
	&		&-{\rm div}\, a(Dm^\delta_\lambda)+(\xi(z)+\lambda\hat{\xi}_\rho)(m^\delta_\lambda)^{p-1} \nonumber \\
	&\leq	&(\xi(z)+\lambda\hat{\xi}_{\rho})m^{p-1}_\lambda + \chi(\delta)\ \mbox{with}\ \chi(\delta)\rightarrow 0^+\ \mbox{as}\ \delta\rightarrow0^+ \nonumber \\
	&\leq	&\xi(z)m^{p-1}_\lambda+(1+\lambda\hat{\xi}_\rho)m^{p-1}_\lambda+\chi(\delta) \nonumber \\
	&\leq	&\overline{\lambda}f(z,m_\lambda)+\lambda\hat{\xi}_\rho m^{p-1}_\lambda+\chi(\delta )\  \mbox{(see (\ref{eq50}))} \nonumber \\
	&<		&\lambda f(z,m_\lambda)-(\lambda-\overline{\lambda})f(z,u_\lambda)+\lambda\hat{\xi}_\rho m^{p-1}_\lambda+\chi(\delta)\nonumber \\
	&	& (\mbox{since}\ \lambda>\overline{\lambda}, \mbox{see hypothesis H(f)(iv)}) \nonumber \\
	&\leq	&\lambda f(z,m_\lambda)+\lambda\hat{\xi}_\rho m^{p-1}_\lambda-(\lambda-\overline{\lambda})\eta_s+\chi(\delta)\ \mbox{with}\ s=m_\lambda > 0\nonumber \\
	&	& \mbox{(see hypothesis H(f)(iv))} \nonumber \\
	&\leq & \lambda f(z,m_\lambda)+\lambda\hat{\xi}_\rho m_\lambda(\vartheta)-\vartheta\ \mbox{for some}\ \vartheta>0\ \mbox{and all}\ \delta>0\ \mbox{small} \nonumber \\
	&\leq & \lambda f(z,u_\lambda)+\lambda\hat{\xi}_\rho u_\lambda-\vartheta\ \mbox{(see hypothesis}\ H(f)(v)) \nonumber \\
	&< &\lambda f(z,u_\lambda)+\lambda\hat{\xi}_\rho u_\lambda \nonumber \\
	&= &-{\rm div}\,a(Du_\lambda)+(\xi(z)+\lambda\hat{\xi}_\rho)u^{p-1}_\lambda\ \mbox{for almost all}\ z\in\Omega\ \mbox{(recall that}\ u_\lambda\in S^\lambda_+).
\end{eqnarray}

	If $\beta=0$ (Neumann problem), then by acting on (\ref{eq51}) with $(m^\delta_\lambda-u_\lambda)^+\in W^{1,p}(\Omega)$ we obtain
	$$m^\delta_\lambda\leq u_\lambda\ \mbox{for}\ \delta>0\ \mbox{small},$$
	a contradiction to the definition of $m_\lambda$.
	
	If $\beta\neq 0$, then from the boundary condition we infer that $\Sigma_0=\{z\in\partial\Omega:u_\lambda(z)=m_\lambda\}\neq\partial\Omega$. Then from (\ref{eq51}) and Proposition \ref{prop7} we have
	$$u_\lambda-m_\lambda\in {\rm int}\, C^*_+\left(\Sigma_0\right)\,,$$
	which again contradicts the definition of $m_\lambda$.
	
	So, it follows that $\lambda\notin\mathcal{L}$ and we have $\lambda^*=\sup\mathcal{L}\leq\overline{\lambda}<\infty$.
\end{proof}

In what follows, for every $\lambda>0,\ \varphi_\lambda:W^{1,p}(\Omega)\rightarrow\RR$ is the $C^1-$energy (Euler) functional for problem (\ref{eqp}) defined by
$$\varphi_\lambda(u)=\int_\Omega G(Du)dz+\frac{1}{p}\int_\Omega\xi(z)|u|^pdz+\frac{1}{p}\int_{\partial\Omega}\beta(z)|u|^pd\sigma-\lambda\int_\Omega F(z,u)dz$$
for all $u\in W^{1,p}(\Omega)$.
\begin{prop}\label{prop16}
	If hypotheses $H(a),\,H(\xi),\,H(\beta),\,H_0,\,H(f)$ hold, then $\lambda^*\in\mathcal{L}$.
\end{prop}
\begin{proof}
	Let $\{\lambda_n\}_{n\geq1}\subseteq\mathcal{L}$ be an increasing sequence such that $\lambda_n\rightarrow\lambda^-$. We can find $u_n\in S^{\lambda_n}_+$ $(n\in\NN)$ such that
	\begin{equation}\label{eq52}
		\varphi_{\lambda_n}(u_n)<0\ \mbox{for all}\ n\in\NN
	\end{equation}
	(see the proof of Proposition \ref{prop13}).
	
	Also, we have
	\begin{equation}\label{eq53}
		\left\langle A(u_n),h\right\rangle + \int_\Omega\xi(z)u^{p-1}_nhdz+\int_{\partial\Omega}\beta(z)u^{p-1}_nhd\sigma =\lambda_n\int_\Omega f(z,u_n)hdz
	\end{equation}
	for all $h\in W^{1,p}(\Omega)$, all $n\in\NN$.
	\begin{claim}\label{c5}
		$\{u_n\}_{n\geq1}\subseteq W^{1,p}(\Omega)$ is bounded.	
	\end{claim}
	
	Arguing by contradiction, suppose that the claim is not true. Then we may assume that
	$$||u_n||\rightarrow+\infty\,.$$

	From (\ref{eq52}) we have
	\begin{equation}\label{eq54}
		\int_\Omega pG(Du_n)dz+\int_\Omega\xi(z)u^p_ndz+\int_{\partial\Omega}\beta(z)u_n^pd\sigma-\lambda_n\int_\Omega pF(z,u_n)dz<0\ \mbox{for all}\ n\in\NN\,.
	\end{equation}
	
	On the other hand, if in (\ref{eq53}) we choose $h=u_n\in W^{1,p}(\Omega)$, then
	\begin{equation}\label{eq55}
		-\int_\Omega(a(Du_n),Du_n)_{\RR^N}dz-\int_\Omega\xi(z)u^p_ndz-\int_{\partial\Omega}\beta(z)u^p_nd\sigma+\lambda_n\int_\Omega f(z,u_n)u_ndz=0\,.
	\end{equation}	
	
	We add (\ref{eq54}), (\ref{eq55}) and obtain
	\begin{eqnarray}\label{eq56}
		&&\int_\Omega[pG(Du_n)-(a(Du_n),Du_n)_{\RR^N}]dz+\lambda_n\int_\Omega e(z,u_n)dz<0\ \mbox{for all}\ n\in\NN \nonumber \\
		&\Rightarrow &\lambda_n\int_\Omega e(z,u_n)dz\leq c_{21}\ \mbox{for some}\ c_{21}>0,\ \mbox{all}\ n\in\NN.
	\end{eqnarray}
	
	Let $y_n=\frac{u_n}{||u_n||},\ n\in\NN$. Then
	$$||y_n||=1,\ y_n\ge\ 0\ \mbox{for all}\ n\in\NN.$$
	
	So, we may assume that
	\begin{equation}\label{eq57}
		y_n\stackrel{w}{\rightarrow}y\ \mbox{in}\ W^{1,p}(\Omega)\ \mbox{and}\ y_n\rightarrow y\ \mbox{in}\ L^r(\Omega)\ \mbox{and in}\ L^p(\partial\Omega),\ y\geq0.
	\end{equation}
	
	First assume that $y\neq0$ and let $E=\{z\in\Omega:y(z)\neq0\}$. We have $|E|_N>0$ and so
	$$u_n(z)\rightarrow+\infty\ \mbox{for almost all}\ z\in E.$$
	
	Hypothesis $H(f)(ii)$ implies that
	\begin{equation}\label{eq58}
		\frac{F(z,u_n)}{||u_n||^p}=\frac{F(z,u_n)}{u^p_n}y^p_n\rightarrow+\infty\ \mbox{for almost all}\ z\in E.
	\end{equation}
	
	From (\ref{eq58}) and Fatou's lemma (hypothesis $H(f)(ii)$ permits its use), we have
	\begin{equation}\label{eq59}
		\frac{1}{||u_n||^p}\int_E F(z,u_n)dz\rightarrow+\infty .
	\end{equation}
	
	Then
	\begin{eqnarray}\label{eq60}
		\int_\Omega\frac{F(z,u_n)}{||u_n||^p}dz &=& \int_E\frac{F(z,u_n)}{||u_n||^p}dz+\int_{\Omega\backslash E}\frac{F(z,u_n)}{||u_n||^p}dz \nonumber \\
		&\geq &\int_E\frac{F(z,u_n)}{||u_n||^p}dz\ \mbox{for almost all}\ n\in\NN\ \mbox{(since}\ F\geq0) \nonumber \\
	\Rightarrow\int_\Omega\frac{F(z,u_n)}{||u_n||^p}dz\rightarrow+\infty\ &\mbox{as}&\ n\rightarrow\infty\ \mbox{(see (\ref{eq59}))}.
	\end{eqnarray}
	
	Hypothesis $H(f)(iii)$ implies that
	\begin{eqnarray}\label{eq61}
		&&0\leq e(z,x)+d(z)\ \mbox{for almost all}\ z\in \Omega,\ \mbox{all}\ x\geq0, \nonumber \\
		&\Rightarrow &pF(z,x)-d(z)\leq f(z,x)x\ \mbox{for almost all}\ z\in\Omega,\ \mbox{all}\ x\geq 0.
	\end{eqnarray}
	
	From (\ref{eq55}), (\ref{eq61}) and hypothesis $H(a)(iv)$, we have
	\begin{eqnarray}\label{eq62}
		\lambda_n\int_\Omega pF(z,u_n)dz & \leq & \int_\Omega pG(Du_n)dz+\int_\Omega\xi(z)u^p_ndz+\int_{\partial\Omega}\beta(z)u^p_nd\sigma+c_{22}\nonumber \\
		& & \mbox{for some}\ c_{22}>0,\ \mbox{all}\ n\in\NN, \nonumber \\
		\Rightarrow\lambda_n\int_\Omega\frac{pF(zu_n)}{||u_n||^p}dz &\leq & \int_\Omega\frac{pG(Du_n)}{||u_n||^p}dz+\int_\Omega\xi(z)y^p_ndz+\int_{\partial\Omega}\beta(z)y^p_nd\sigma+\frac{c_{22}}{||u_n||^p}\nonumber \\
		&\leq & pc_5\left(\frac{1}{||u_n||^p}+||Dy_n||^p_p\right)+\int_\Omega\xi(z)y^p_ndz+\int_{\partial\Omega}\beta(z)y^p_nd\sigma\nonumber \\
		& & +\frac{c_{22}}{||u_n||^p} \nonumber \\
		&\leq & c_{23}\ \mbox{for some}\ c_{23}>0,\ \mbox{all}\ n\in\NN.
	\end{eqnarray}
	
	Comparing (\ref{eq60}) and (\ref{eq62}), we have a contradiction.
	
	Next assume that $y=0$. For $\mu>0$, we set
	$$v_n=(p\mu)^{^1/_p}y_n\in W^{1,p}(\Omega)\ \mbox{for all}\ n\in \NN.$$
	
	Note that
	\begin{eqnarray}
		&&v_n\rightarrow0\ \mbox{in}\ L^r(\Omega)\ \mbox{(see (\ref{eq57}) and recall that}\ y=0),\label{eq63} \\
		&\Rightarrow &\int_\Omega F(z,v_n)dz\rightarrow 0\ \mbox{(see hypothesis H(f)(i))}.\label{eq64}
	\end{eqnarray}
	
	Since $||u_n||\rightarrow\infty$, we can find $n_0\in\NN$ such that
	\begin{equation}\label{eq65}
		(p\mu)^{^1/_p}\frac{1}{||u_n||}\leq1\ \mbox{for all}\ n\geq n_0.
	\end{equation}
	
	Consider the $C^1-$functional $\tilde{\psi}_{\lambda_n}:W^{1,p}(\Omega)\rightarrow\RR$ defined by
	\begin{equation*}
		\tilde{\psi}_{\lambda_n}(u)=\frac{c_1}{p(p-1)}||Du||^p_p+\frac{1}{p}\int_\Omega\xi(z)|u|^pdz+\frac{1}{p}\int_{\partial\Omega}\beta(z)|u|^pd\sigma-\lambda_n\int_\Omega F(z,u)dz
	\end{equation*}
	for all $u\in W^{1,p}(\Omega).$
	
	Let $t_n \in[0,1]$ be such that
	\begin{equation}\label{eq66}
		\tilde{\psi}_{\lambda_n}(t_n u_n)=\max\left[\tilde{\psi}_{\lambda_n}(tu_n):0\leq t\leq1\right].
	\end{equation}
	
	From (\ref{eq65}) and (\ref{eq66}) it follows that
	\begin{eqnarray}\label{eq67}
		\tilde{\psi}_{\lambda_n}(t_nu_n) & \geq & \tilde{\psi}_{\lambda_n}(v_n) \nonumber \\
		&=&\mu\left[\frac{c_1}{p-1}||Dy_n||^p_p+\int_\Omega\xi(z)y^p_ndz+
\int_{\partial\Omega}\beta(z)y^p_nd\sigma\right]-\lambda_n\int_\Omega F(z,v_n)dz \nonumber \\
		&\geq &\mu c_{24}-\lambda^*\int_\Omega F(z,v_n)dz\ \mbox{for some}\ c_{24}>0,\ \mbox{all}\ n\in\NN \nonumber \\
		&&\mbox{(see hypothesis $H_0$, Lemmata \ref{lem5}, \ref{lem6} and recall that}\ F\geq0,\lambda_n\leq\lambda^*) \nonumber \\
		&\geq &\mu\frac{c_{24}}{2}>0\ \mbox{for all}\ n\geq n_1 \geq n_0\ \mbox{(see (\ref{eq64}))}.
	\end{eqnarray}
	
	But $\mu>0$ is arbitrary. So, from (\ref{eq67}) we infer that
	\begin{equation}\label{eq68}
		\tilde{\psi}_{\lambda_n}(t_nu_n)\rightarrow+\infty\ \mbox{as}\ n\rightarrow+\infty\,.
	\end{equation}
	
	Note that
	\begin{equation}\label{eq69}
		\tilde{\psi}_{\lambda_n}(0)=0\ \mbox{and}\ \tilde{\psi}_{\lambda_n}(u_n)<0\ \mbox{for all}\ n\in\NN
	\end{equation}
	(see (\ref{eq52}) and by Corollary \ref{cor3}, $\tilde{\psi}_{\lambda_n}\leq\varphi_{\lambda_n}$ for all $n\in\NN$).
	
	Then from (\ref{eq68}) and (\ref{eq69}) it follows that
	\begin{equation}\label{eq70}
		t_n\in(0,1)\ \mbox{for all}\ n\geq n_2.
	\end{equation}
	
	So, from (\ref{eq66}) and (\ref{eq70}) we have
	\begin{eqnarray}\label{eq71}
&\frac{d}{dt}\tilde{\psi}_{\lambda_n}(tu_n)|_{t=t_n}=0\ \mbox{for all}\ n\geq n_2,\nonumber\\
		\Rightarrow & \langle \tilde{\psi}'_{\lambda_n}(t_nu_n),t_nu_n\rangle  = 0\ \mbox{for all}\ n\geq n_2\ \mbox{(by the chain rule)} \nonumber \\
		\Rightarrow & \frac{c_1}{p-1}||D(t_nu_n)||^p_p  +\int_\Omega\xi(z)(t_nu_n)^pdz+\int_{\partial\Omega}\beta(z)(t_nu_n)^pd\sigma=\nonumber\\
		&\lambda_n\int_\Omega f(z,t_nu_n)(t_nu_n)dz\ \mbox{for all}\ n\geq n_2, \nonumber \\
		\Rightarrow & p\tilde{\psi}_{\lambda_n}(t_nu_n)  +\lambda_n\int_\Omega pF(z,t_nu_n)dz=\lambda_n\int_\Omega f(z,t_nu_n)(t_nu_n)dz \  \mbox{for all}\ n\geq n_2, \nonumber \\
		\Rightarrow & p\tilde{\psi}_{\lambda_n}(t_nu_n)  \leq\lambda_n\int_\Omega e(z,t_nu_n)dz \nonumber \\
		& \leq\lambda^*\int_\Omega e(z,t_nu_n)dz\ \mbox{(since}\ \lambda_n\leq\lambda^*\ \mbox{for all}\ n\in\NN\ \mbox{and}\ e\geq 0) \nonumber \\
		& \leq \lambda^*\int_\Omega e(z,u_n)dz+\lambda^*||d||_1\ \mbox{(see (\ref{eq70}) and hypothesis H(f)(iii))}\nonumber \\
		& \leq M_4\ \mbox{for some}\ M_4>0,\ \mbox{all}\ n\geq n_2.
	\end{eqnarray}
	
	Comparing (\ref{eq68}) and (\ref{eq71}) again we have a contradiction.
	
	This proves the claim.
	
	On account of  Claim~\ref{c5}, we may assume that
	\begin{equation}\label{eq72}
		u_n\stackrel{w}{\rightarrow}u^*\ \mbox{in}\ W^{1,p}(\Omega)\ \mbox{and}\ u_n\rightarrow u^*\ \mbox{in}\ L^r(\Omega)\ \mbox{and in}\ L^p(\partial\Omega).
	\end{equation}
	
	In (\ref{eq53}) we choose $h=u_n-u^*\in W^{1,p}(\Omega)$, pass to the limit as $n\rightarrow\infty$ and use (\ref{eq72}). We obtain
	\begin{eqnarray}\label{eq73}
		&&\lim\limits_{n\rightarrow\infty}\langle A(u_n),u_n-u^*\rangle=0, \nonumber \\
		&\Rightarrow &u_n\rightarrow u^*\ \mbox{in}\ W^{1,p}(\Omega)\ \mbox{(see Proposition \ref{prop4}).}
	\end{eqnarray}
	
	So, if in (\ref{eq53}) we pass to the limit as $n\rightarrow\infty$ and use (\ref{eq73}), then
	\begin{eqnarray}\label{eq74}
		&&\langle A(u^*),h\rangle+\int_\Omega\xi(z)(u^*)^{p-1}hdz+\int_{\partial\Omega}\beta(z)(u^*)^{p-1}hd\sigma=\lambda^*\int_\Omega f(z,u^*)hdz \nonumber \\
		&& \mbox{for all}\ h\in W^{1,p}(\Omega), \nonumber \\
		&\Rightarrow &-{\rm div}\,a(Du^*(z))+\xi(z)u^*(z)^{p-1}=\lambda^*f(z,u^*(z))\ \mbox{for almost all}\ z\in\Omega, \nonumber \\
		&&\frac{\partial u^*}{\partial n_a}+\beta(z)(u^*)^{p-1}=0\ \mbox{on}\ \partial\Omega\ \mbox{(see Papageorgiou and R\u adulescu \cite{27})}.
	\end{eqnarray}
	
	We know that
	$$\overline{u}_{\lambda_1}\leq u_n\ \mbox{for all}\ n\in\NN$$
	(see Claim \ref{claim2} in the proof of Proposition \ref{prop10} and use the fact that $\lambda\mapsto\overline{u}_\lambda$ is nondecreasing from $(0,+\infty)$ into $C^1(\overline{\Omega}))$. Hence in the limit as $n\rightarrow\infty$, we obtain
	\begin{eqnarray*}
		&&\overline{u}_{\lambda_1}\leq u^*, \\
		&\Rightarrow & u^*\in S^{\lambda^*}_+\ \mbox{(see (\ref{eq74})) and so}\ \lambda^*\in\mathcal{L}.
	\end{eqnarray*}
\end{proof}
\begin{prop}\label{prop17}
	If hypotheses $H(a),H(\xi),H(\beta),H_0,H(f)$ hold and $\lambda\in(0,\lambda^*)$, then problem (\ref{eqp}) has at least two positive solutions
	$$v_\lambda,\hat{u}_\lambda\in D_+\ \mbox{with}\ \hat{u}_\lambda-u_\lambda\in C_+\backslash\{0\}.$$
\end{prop}

\begin{proof}
	From Proposition \ref{prop16} we know that $\lambda^*\in\mathcal{L}$. So, we can find $u^*\in S^{\lambda^{*}}_+ \subseteq D_+$. Invoking Corollary \ref{cor14}, we can find $u_\lambda\in S^\lambda_+\subseteq D_+$ such that
	\begin{equation}\label{eq75}
		u^*-u_\lambda\in {\rm int}\, C^*_+\left(\Sigma_0\right)
	\end{equation}
	with $\Sigma_0=\{z\in\partial\Omega:u_\lambda(z)=u^*(z)\}$. Moreover, from the proof of Proposition \ref{prop13} we know that $u_\lambda$ is a global minimizer of the functional $w_\lambda$.
	
	Using $u_\lambda\in S^\lambda_+\subseteq D_+$, we introduce the following truncation of the reaction term in problem (\ref{eqp}):
	\begin{equation}\label{eq76}
		\vartheta_\lambda(z,x)=\left\{\begin{array}{ll}
			\lambda f(z,u_\lambda(z)) 	& \mbox{if}\ x\leq u_\lambda(z) \\
			\lambda f(z,x)				& \mbox{if}\ u_\lambda(z)<x.
		\end{array}\right.
	\end{equation}
	
	This is a Carath\'eodory function. We set $\Theta_\lambda(z,x)=\int^x_0\vartheta_\lambda(z,s)ds$ and consider the $C^1-$functional $\hat{\varphi}_\lambda:W^{1,p}(\Omega)\rightarrow\RR$ defined by
	\begin{equation*}
		\hat{\varphi}_\lambda(u)=\int_\Omega G(Du)dz+\frac{1}{p}\int_\Omega\xi(z)|u|^pdz+\frac{1}{p}\int_{\partial\Omega}\beta(z)|u|^pd\sigma-\int_\Omega\Theta_\lambda(z,u)dz
	\end{equation*}
	for all $u\in W^{1,p}(\Omega)$.
	
	From (\ref{eq76}) it is clear $\vartheta_\lambda(z,\cdot)$ has the same asymptotic behavior for $x\rightarrow+\infty$ as $f(z,\cdot)$. So, reasoning as in Claim~\ref{c5} in the proof of Proposition \ref{prop16}, we show that
	\begin{equation}\label{eq77}
		\hat{\varphi}_\lambda\ \mbox{satisfies the C-condition}.
	\end{equation}
	\begin{claim}\label{c6}
		$K_{\hat{\varphi}_\lambda}\subseteq[u_\lambda)\cap D_+=\{u\in D_+:u_\lambda(z)\leq u(z)\ \mbox{for all}\ z\in\overline{\Omega}\}$.	
	\end{claim}
	
	Let $u\in K_{\hat{\varphi}_\lambda}$. Then
	\begin{equation}\label{eq78}
		\langle A(u),h\rangle+\int_\Omega\xi(z)|u|^{p-2}uhdz+\int_{\partial\Omega}\beta(z)|u|^{p-2}uhd\sigma=\int_\Omega\vartheta_\lambda(z,u)hdz
	\end{equation}
	for all $h\in W^{1,p}(\Omega).$
	
	In (\ref{eq78}) we choose $h=(u_\lambda-u)^+\in W^{1,p}(\Omega)$. Then
	\begin{eqnarray*}
		&&\langle A(u),(u_\lambda-u)^+\rangle+\int_\Omega\xi(z)|u|^{p-2}u(u_\lambda-u)^+dz+\int_{\partial\Omega}\beta(z)|u|^{p-2}u(u_\lambda-u)^+d\sigma \\
		&=&\int_\Omega\lambda f(z,u_\lambda)(u_\lambda-u)^+dz\ \mbox{(see (\ref{eq76}))} \\
		&=&\langle A(u_\lambda),(u_\lambda-u)^+\rangle+\int_\Omega\xi(z)u^{p-1}_\lambda(u_\lambda-u)^+dz+\int_{\partial\Omega}\beta(z)u^{p-1}_\lambda(u_\lambda-u)^+d\sigma \\
		&& \mbox{(since}\ u_\lambda\in S^\lambda_+)\\
		&\Rightarrow & u_\lambda\leq u.
	\end{eqnarray*}
	
	As before, the nonlinear regularity theory implies that $u\in D_+$.
	
	This proves Claim~\ref{c6}.
	
	Claim~\ref{c6} allows us to assume that
	\begin{equation}\label{eq79}
		K_{\hat{\varphi}_\lambda}\cap[u_\lambda,u^*]=\{u_\lambda\}.
	\end{equation}
	Indeed, otherwise we already have a second positive smooth (due to nonlinear regularity) solution of problem (\ref{eqp}) which is bigger than $u_\lambda$ and so we are done.
	
	We consider the following truncation of $\vartheta_\lambda(z,\cdot):$
	\begin{equation}\label{eq80}
		\tilde{\vartheta}_\lambda(z,x)=\left\{\begin{array}{ll}
			\vartheta_\lambda(z,x) 		& \mbox{if}\ x\leq u^*(z) \\
			\vartheta_\lambda(z,u^*(z))	& \mbox{if}\ u^*(z)<x.
		\end{array}\right.
	\end{equation}
	
	This is a Carath\'eodory function. We set $\tilde{\Theta}_\lambda(z,x)=\int^x_0\tilde{\vartheta}_\lambda(z,s)ds$ and consider the $C^1-$functional $\tilde{\varphi}_\lambda:W^{1,p}(\Omega)\rightarrow\RR$ defined by
	\begin{equation*}
		\tilde{\varphi}_\lambda(u)=\int_\Omega G(Du)dz+\frac{1}{p}\int_\Omega\xi(z)|u|^pdz+\frac{1}{p}\int_{\partial\Omega}\beta(z)|u|^pd\sigma-\int_\Omega\tilde{\Theta}_\lambda(z,u)dz
	\end{equation*}
	for all $u\in W^{1,p}(\Omega)$.
	
	Using (\ref{eq80}) we easily show that
	\begin{equation}\label{eq81}
		K_{\tilde{\varphi}_\lambda}\subseteq[u_\lambda,u^*]\cap D_+.
	\end{equation}

	From (\ref{eq80}) it is clear that $\tilde{\varphi}_\lambda$ is coercive. Also, it is sequentially weakly lower semicontinuous. So, we can find $\tilde{u}_\lambda\in W^{1,p}(\Omega)$ such that
	\begin{eqnarray}\label{eq82}
		&&\tilde{\varphi}_\lambda(\tilde{u}_\lambda)=\inf[\tilde{\varphi}_\lambda(u):u\in W^{1,p}(\Omega)], \nonumber \\
		&\Rightarrow &\tilde{u}_\lambda\in K_{\tilde{\varphi}_\lambda}\subseteq[u_\lambda,u^*]\cap D_+\ \mbox{(see (\ref{eq81})).}
	\end{eqnarray}
	
	From (\ref{eq76}) and (\ref{eq80}), we see that
	\begin{eqnarray*}
		&&\tilde{\varphi}'_\lambda|_{[0,u^*]}=\hat{\varphi}'_\lambda|_{[0,u^*]}, \\
		&\Rightarrow &\tilde{u}_\lambda\in K_{\hat{\varphi}_\lambda}\ \mbox{(see (\ref{eq82}))}, \\
		&\Rightarrow &\tilde{u}_\lambda=u_\lambda\ \mbox{(see (\ref{eq79}), (\ref{eq82}))}.
	\end{eqnarray*}
	
	Then from (\ref{eq75}) we infer that for $\Sigma_0=\{z\in\partial\Omega:u_\lambda(z)=u^*(z)\}$ we have
	\begin{eqnarray}\label{eq83}
		&&u_\lambda\ \mbox{is a}\ C^1_*(\Omega)-\mbox{minimizer of}\ \hat{\varphi}_\lambda, \nonumber \\
		&\Rightarrow &u_\lambda\ \mbox{is a}\ W^{1,p}_*(\Omega)-\mbox{minimizer of}\ \hat{\varphi}_\lambda\ \mbox{(see Proposition \ref{prop8})}.
	\end{eqnarray}
	
	Without any loss of generality, we may assume that
	\begin{equation}\label{eq84}
		K_{\hat{\varphi}_\lambda}\ \mbox{is finite}.
	\end{equation}
	
	Otherwise Claim~\ref{c6} and (\ref{eq76}) imply that we already have a whole sequence of distinct smooth solutions of (\ref{eqp}) bigger than $u_\lambda$ and so we are done. Then (\ref{eq84}) implies that we can find $\rho\in(0,1)$ small such that
	\begin{equation}\label{eq85}
		\hat{\varphi}_\lambda(u_\lambda)<\inf\left[ \hat{\varphi}_\lambda(u_\lambda+h):||h||\leq\rho,h\in W^{1,p}_*(\Omega) \right]=\hat{m}^\lambda_\rho\,.
	\end{equation}
	
	In addition, hypothesis $H(f)(ii)$ implies that for all $h\in {\rm int}\,C^*_+\left(\Sigma_0\right)$, we have
	\begin{equation}\label{eq86}
		\hat{\varphi}_\lambda(u_\lambda+th)\rightarrow-\infty\ \mbox{as}\ t\rightarrow+\infty\,.
	\end{equation}
	
	From (\ref{eq77}), (\ref{eq85}), (\ref{eq86}) we see that we can apply Theorem \ref{th1} (the mountain pass theorem) on the affine space (manifold) $Y=u_\lambda+W^{1,p}_*(\Omega)$) and find $\hat{u}_\lambda \in Y$ such that
	\begin{eqnarray}\label{eq87}
		&&\langle \hat{\varphi}'_\lambda(\hat{u}_\lambda),h\rangle=0\ \mbox{for all}\ h\in W^{1,p}_*(\Omega),\ \hat{m}^\lambda_\rho\leq\hat{\varphi}_\lambda(\hat{u}_\lambda)\ \mbox{(see (\ref{eq85}))}, \\
		&\Rightarrow &u_\lambda\leq\hat{u}_\lambda\ \mbox{(by choosing}\ h=(u_\lambda-\hat{u}_\lambda)^+\in W^{1,p}_*(\Omega)).\nonumber
	\end{eqnarray}
	
	Also, using the nonlinear Green's identity on the space $W^{1,p}_*(\Omega)$ (see Casas and Fernandez \cite{7} and Kenmochi \cite{22}) from (\ref{eq87}) we infer that
	$$\hat{u}_{\lambda}\in D_+\ \mbox{is a solution of}\ \eqref{eqp}\ (\lambda\in(0,\lambda^*)).$$
	
	Moreover, from (\ref{eq85}) we have
	$$\hat{u}_\lambda-u_\lambda\in C_+\backslash\{0\}.$$
\end{proof}

Summarizing the results of this section, we can formulate the following bifurcation-type result.
\begin{theorem}\label{th18}
	If hypotheses $H(a),H(\xi),H(\beta),H_0,H(f)$ hold, then there exists $\lambda^*>0$ such that
	\begin{itemize}
		\item [(a)] for all $\lambda\in(0,\lambda^*)$ problem (\ref{eqp}) has at least two positive solutions
				$$u_\lambda,\hat{u}_\lambda\in D_+\ \mbox{with}\ \hat{u}_\lambda-u_\lambda\in C_+\backslash\{0\};$$
		\item [(b)] for $\lambda=\lambda^*$ problem (\ref{eqp}) has at least one positive solution
				$$u^*\in D_+;$$
		\item [(c)] for $\lambda>\lambda^*$ problem (\ref{eqp}) has no positive solution.
	\end{itemize}
\end{theorem}

\section{Big, Small and Minimal Positive Solutions}

In this section we show that as $\lambda\rightarrow 0^+$, we can produce positive solutions of problem (\ref{eqp}) which have $W^{1,p}(\Omega)-$norm which is arbitrarily big and arbitrarily small. Moreover, we show that for every $\lambda\in (0,\lambda^*)$, problem (\ref{eqp}) admits a smallest positive solution $u^*_\lambda\in D_+$ and study the monotonicity and continuity properties of the map $\lambda\mapsto u^*_\lambda$.
\begin{theorem}\label{th19}
	If hypotheses $H(a),H(\xi),H(\beta),H_0,H(f)$ hold and $\lambda_n\rightarrow0^+$, then we can find positive solutions
	$$\hat{u}_n=\hat{u}_{\lambda_n}\in S^{\lambda_n}_+\subseteq D_+\ \mbox{and}\ u_n=u_{\lambda_n}\in S^{\lambda_n}_+\subseteq D_+\ \mbox{for all}\ n\in\NN$$
	such that $||\hat{u}_n||\rightarrow+\infty$ and $||u_n||\rightarrow0$ as $n\rightarrow\infty$.
\end{theorem}

\begin{proof}
	From (\ref{eq17}) we have
	\begin{equation}\label{eq88}
		F(z,x)\leq\frac{\hat{\eta}}{\tau}x^\tau+\frac{c_{10}}{r}x^r\ \mbox{for almost all}\ z\in\Omega,\ \mbox{all}\ x\geq0.
	\end{equation}
	
	Then for all $u\in W^{1,p}(\Omega)$ we have
	\begin{eqnarray}\label{eq89}
		\varphi_{\lambda_n}(u) & \geq & \frac{c_1}{p(p-1)}||Du_n||^p_p+\frac{1}{p}\int_\Omega\xi(z)|u|^pdz+\frac{1}{p}\int_{\partial\Omega}\beta(z)|u|^pd\sigma-\nonumber\\
		&&\frac{\lambda_n\hat{\eta}}{\tau}||u^+||^\tau_\tau-\frac{\lambda_n c_{10}}{r}||u^+||^r_r \nonumber \\
		&& \mbox{(see Corollary \ref{cor3} and (\ref{eq88}))} \nonumber \\
		& \geq & c_{25}||u||^p-\lambda_n c_{26}(||u||^\tau+||u||^r)\ \mbox{for some}\ c_{25},c_{26}>0,\ \mbox{all}\ n\in\NN \\
		&& \mbox{(se hypothesis $H_0$ and Lemmata \ref{lem5}, \ref{lem6})}.\nonumber
	\end{eqnarray}
	
	Let $||u||=\lambda^{-\alpha}_n$ with $\alpha>0$. We set
	$$k(\lambda_n)=c_{25}\lambda^{-\alpha p}_n-c_{26}(\lambda^{1-\alpha\tau}_n+\lambda^{1-\alpha r}_n),\ n\in\NN.$$
	
	We choose $\alpha\in \left(0,\frac{1}{r-p}\right)$ (recall that $r>p$). Then we have
	$$-\alpha p<1-\alpha r<1-\alpha \tau\ \mbox{(recall that}\ \tau<p<r).$$
	
	So, we see that
	\begin{equation}\label{eq90}
		k(\lambda_n)\rightarrow+\infty\ \mbox{as}\ n\rightarrow\infty\ \mbox{(recall that}\ \lambda_n\rightarrow0^+).
	\end{equation}
	
	Then from (\ref{eq89}) and (\ref{eq90}), we infer that there exists $n_1\in\NN$ such that
	\begin{equation}\label{eq91}
		\varphi_{\lambda_n}(u)\geq k(\lambda_n)>0=\varphi_{\lambda_n}(0)\ \mbox{for all}\ n\geq n_1\ \mbox{and all}\ ||u||=\lambda^{-\alpha}_n.
	\end{equation}
	
	Hypothesis $H(f)(ii)$ implies that if $u\in D_+$, then
	\begin{equation}\label{eq92}
		\varphi_{\lambda_n}(tu)\rightarrow-\infty\ \mbox{as}\ t\rightarrow+\infty,\ \mbox{for all}\ n\in\NN.
	\end{equation}
	
	Moreover, as in  Claim~\ref{c5} in the proof of Proposition \ref{prop16}, we can check that
	\begin{equation}\label{eq93}
		\varphi_{\lambda_n}(\cdot)\ \mbox{satisfies the C-condition for all}\ n\in\NN.
	\end{equation}
	
	Then (\ref{eq91}), (\ref{eq92}), (\ref{eq93}) permit the use of Theorem \ref{th1} (the mountain pass theorem). So, we can find $\hat{u}_n\in W^{1,p}(\Omega)$ such that
	\begin{eqnarray*}
		&&\hat{u}_n\in K_{\varphi_{\lambda_n}}\ \mbox{and}\ k(\lambda_n)\leq\varphi_\lambda(\hat{u}_n)\leq c_{27}(1+||\hat{u}_n||^r)\ \mbox{for some}\ c_{27}>0,\ \mbox{all}\ n\geq n_1 \\
		&& \mbox{(see hyothesis}\ H(f)(i)),\\
		&\Rightarrow &\hat{u}_n\in S^{\lambda_n}_+\subseteq D_+\ \mbox{for all}\ n\in\NN\ \mbox{and}\ ||\hat{u}_n||\rightarrow\infty\ \mbox{(see (\ref{eq90}))}.
	\end{eqnarray*}
	
	Next let $\zeta\in\left(0,\frac{1}{p}\right)$ and consider $||u||=\lambda^\zeta_n$. Then from (\ref{eq89}) we have
	\begin{eqnarray*}
		&&\varphi_{\lambda_n}(u)\geq c_{25}\lambda^{\zeta p}_n -c_{26}\left( \lambda^{\zeta\tau+1}_n+\lambda^{\zeta r+1}_n \right) \\
		&&=\lambda_n\left[c_{25}\lambda^{\zeta p-1}_n-c_{26}\left(\lambda^{\zeta\tau}_n+\lambda^{\zeta r}_n\right)\right].
	\end{eqnarray*}
	
	Let $k_0(\lambda_n)=c_{25}\lambda^{\zeta p-1}_n-c_{26}\left(\lambda^{\zeta\tau}_n+\lambda^{\zeta r}_n\right)$. Since $\zeta p-1<0$ and $\lambda_n\rightarrow0^+$, we infer that
	$$k_0(\lambda_n)\rightarrow+\infty\ \mbox{as}\ n\rightarrow+\infty\,.$$
	
	So, we can find $n_2\in\NN$ such that
	\begin{equation}\label{eq94}
		\varphi_{\lambda_n}(u)\geq \lambda_nk_0(\lambda_n)>0=\varphi_{\lambda_n}(0)\ \mbox{for all}\ n\geq n_2\ \mbox{and all}\ ||u||=\lambda^\zeta_n.
	\end{equation}
	
	Let $\overline{B}_n=\{u\in W^{1,p}(\Omega):||u||\leq\lambda^\zeta_n\},\ n\in\NN$. Hypotheses $H(a)(iv),H(f)(iv)$ and since $\tau<q<p$, imply that for every $n\in\NN$, every $u\in D_+$ and for $t\in(0,1)$ small, we have
	\begin{equation}\label{eq95}
		\varphi_{\lambda_n}(tu)<0,\ ||tu||\leq\lambda^\zeta_n\ \mbox{for all}\ n\in\NN\ \mbox{(see the proof of Proposition \ref{prop11}).}
	\end{equation}
	
	From (\ref{eq94}) and (\ref{eq95}), we see that
	\begin{equation}\label{eq96}
		0<\inf\limits_{\partial B_n}\varphi_{\lambda_n},\ \mbox{and}\ \inf\limits_{\overline{B}_n}\varphi_{\lambda_n}<0\ \mbox{for all}\ n\geq n_2.
	\end{equation}
	
	Let $\bar{c}_n=\inf\limits_{\partial B_n}\varphi_{\lambda_n}-\inf\limits_{\overline{B}_n}\varphi_{\lambda_n}>0$ for $n\geq n_2$ (see (\ref{eq96})). Using the Ekeland variational principle (see, for example, Gasinski and Papageorgiou \cite[pp. 579]{14}), given $\epsilon\in(0,\tau_n)$ ($n\geq n_2$), we can find $u^n_\epsilon\in B_n=\left\{u\in W^{1,p}(\Omega):||u||<\lambda^\zeta_n\right\}$ such that
	\begin{eqnarray}
		\varphi_{\lambda_n}(u^n_\epsilon) & \leq & \inf\limits_{\overline{B_n}}\varphi_{\lambda_n}+\epsilon \label{eq97}\\
		\varphi_{\lambda_n}(u^n_\epsilon) & \leq & \varphi_{\lambda_n}(y)+\epsilon||y-u_n||\ \mbox{for all}\ y\in\overline{B_n},\ n\geq n_2. \label{eq98}
	\end{eqnarray}
	
	Given $t\in W^{1,p}(\Omega)$, for $t>0$ small we have
	$$u^n_\epsilon+th\in \overline{B}_n.$$
	
	So, if in (\ref{eq98}) we choose $y=u^n_\epsilon+th$, then
	\begin{eqnarray}\label{eq99}
		&&-\epsilon||h||\leq\langle\varphi'_{\lambda_n}(u^n_\epsilon),h\rangle\ \mbox{for all}\ h\in W^{1,p}(\Omega), \nonumber \\
		&\Rightarrow &||\varphi'_{\lambda_n}(u^n_\epsilon)||_*\leq\epsilon\ \mbox{for all}\ n\geq n_2.
	\end{eqnarray}
	
	Let $\epsilon_m\rightarrow0^+$ and set $u^n_{\epsilon_m}=u^n_m$ for all $m\in\NN,n\geq n_2$. From (\ref{eq99}) we have
	\begin{equation}\label{eq100}
		\varphi'_{\lambda_n}(u^n_m)\rightarrow0\ \mbox{in}\ W^{1,p}(\Omega)^*\ \mbox{as}\ m\rightarrow\infty,\ n\geq n_2.
	\end{equation}
	
	But from (\ref{eq93}) we know that $\varphi_{\lambda_n}(\cdot)$ satisfies the C-condition. So, from (\ref{eq97}) and (\ref{eq100}) if follows that at least for a subsequence, we have
	\begin{equation}\label{eq101}
		u^n_m\rightarrow u_{\lambda_n}=u_n\ \mbox{in}\ W^{1,p}(\Omega)\ \mbox{as}\ m\rightarrow\infty\,.
	\end{equation}
	
	From (\ref{eq97}) and (\ref{eq101}), we infer that
	\begin{eqnarray*}
		&&\varphi_{\lambda_n}(u_n)=\inf\limits_{\overline{B}_n}\varphi_{\lambda_n}\ \mbox{for all}\ n\geq n_2, \\
		&\Rightarrow & u_n\in B_n\ \mbox{and so}\ u_n\in K_{\varphi_{\lambda_n}}\ \mbox{for all}\ n\geq n_2\ \mbox{(see (\ref{eq96})).}
	\end{eqnarray*}

	Therefore we have
	\begin{eqnarray*}
		&&u_n\in S^\lambda_+\subseteq D_+\ \mbox{and}\ ||u_n||<\lambda^\zeta_n\ \mbox{for all}\ n\geq n_2,\\
		&\Rightarrow & ||u_n||\rightarrow0\ \mbox{as}\ n\rightarrow\infty\ \mbox{(recall that}\ \lambda_n\rightarrow0^+).
	\end{eqnarray*}
\end{proof}

For every $\lambda\in(0,\lambda^*)$ we show that problem (\ref{eqp}) admits a minimal positive solution $u^*_\lambda$ and determine the monotonicity and continuity properties of the map $\lambda\mapsto u^*_\lambda$.
\begin{theorem}\label{th20}
	If hypotheses $H(a),H(\xi),H(\beta),H_0,H(f)$ hold and $\lambda\in(0,\lambda^*)$, then problem (\ref{eqp}) has a smallest positive solution $u^*_\lambda\in S^\lambda_+\subseteq D_+$ and the map
	$\lambda\mapsto u^*_\lambda$ from $(0,\lambda^*)$ into $C^1(\overline{\Omega})$ is
	\begin{itemize}
		\item ``strictly monotone", in the sense that
	$$\vartheta<\lambda\Rightarrow u^*_\lambda-u^*_\vartheta\in {\rm int}\, C^*_+\left(\Sigma_0\right)$$
	with $\Sigma_0=\{z\in\partial\Omega:u^*_\lambda(z)=u^*_\vartheta(z)\}$;\\
		\item ``left continuous", that is, if $\lambda_n\rightarrow \lambda^-<\lambda^*$, then $u_{\lambda_n}\rightarrow u_{\lambda}$ in $C^1(\overline{\Omega})$.
	\end{itemize}
\end{theorem}
\begin{proof}
	From Lemma 3.10 of Hu and Papageorgiou \cite[p. 178]{19}, we know that we can find $\{u_n\}_{n\geq1}\subseteq S^\lambda_+$ such that
	$$\inf S^\lambda_+=\inf\limits_{n\geq 1}u_n,\ u_n\leq\tilde{u}_\lambda\ \mbox{for all}\ n\in\NN\ \mbox{(see the proof of Proposition \ref{prop11})}$$
	Evidently, $\{u_n\}_{n\geq1}\subseteq W^{1,p}(\Omega)$ is bounded and so we may assume that
	\begin{equation}\label{eq102}
		u_n\stackrel{w}{\rightarrow}u^*_\lambda\ \mbox{in}\ W^{1,p}(\Omega)\ \mbox{and}\ u_n\rightarrow u^*_\lambda\ \mbox{in}\ L^r(\Omega)\ \mbox{and in}\ L^p(\partial\Omega).
	\end{equation}
	
	We have for all $h\in W^{1,p}(\Omega)$
	\begin{equation}\label{eq103}
		\langle A(u_n),h\rangle+\int_\Omega\xi(z)u^{p-1}_nhdz+\int_{\partial\Omega}\beta(z)u^{p-1}_nhd\sigma=\lambda\int_\Omega f(z,u_n)hdz\,.
	\end{equation}
	
	In (\ref{eq103}) we choose $h=u_n-u^*_\lambda\in W^{1,p}(\Omega)$. Passing to the limit as $n\rightarrow\infty$ and using (\ref{eq102}), we obtain
	\begin{eqnarray}\label{eq104}
		&&\lim\limits_{n\rightarrow\infty}\langle A(u_n),u_n-u^*_\lambda\rangle=0 \nonumber \\
		&\Rightarrow &u_n\rightarrow u^*_\lambda\ \mbox{in}\ W^{1,p}(\Omega)\ \mbox{(see Proposition \ref{prop4}}).
	\end{eqnarray}
	
	Hence if in (\ref{eq103}) we pass to the limit as $n\rightarrow\infty$ and use (\ref{eq104}), then
	$$\langle A(u^*_\lambda),h\rangle+\int_\Omega\xi(z)(u^*_\lambda)^{p-1} hdz+\int_{\partial\Omega}\beta(z)(u^*_\lambda)^{p-1} hd\sigma=\lambda\int_\Omega f(z,u^*_\lambda)hdz\ \mbox{for all}\ h\in W^{1,p}(\Omega)$$
	$\Rightarrow u^*_\lambda$ is a nonnegative solution of (\ref{eqp}) (see Papageorgiou and R\u adulescu \cite{27}).
	
	Hypotheses $H(f)(i),(iv)$ imply that we can find $c_{28}>0$ such that
	\begin{equation}\label{eq105}
		f(z,x)\geq \hat{\eta}_0 x^{\tau-1}-c_{28}x^{r-1}\ \mbox{for almost all}\ z\in\Omega,\ \mbox{all}\ x\geq0.
	\end{equation}
	
	We consider the following auxiliary Robin problem
	\begin{equation}\label{eqaulam''}\tag*{$(Au_\lambda)^{''}$}
		\left\{ \begin{array}{l}
			-{\rm div}\,a(Du(z))+\xi(z)u(z)^{p-1}=\lambda(\hat{\eta}_0 u(z)^{\tau-1}-c_{28}u(z)^{r-1})\ \mbox{in}\ \Omega,\\
			\frac{\partial u}{\partial n_a}+\beta(z)u^{p-1}=0\ \mbox{on}\ \partial\Omega, \ u> 0,\ \lambda>0.
		\end{array} \right\}
	\end{equation}
	
	As in the proof pf Proposition \ref{prop10} (there we had the auxiliary problem \ref{eqaulam'}), problem \ref{eqaulam''}  has a unique positive solution $\overline{u}^*_\lambda\in D_+$ for all $\lambda>0$ and
	$$\overline{u}^*_\lambda\leq u\ \mbox{for all}\ u\in S^\lambda_+\ \mbox{(see (\ref{eq105})).}$$
	
	So, we have
	\begin{eqnarray*}
	&&\overline{u}^*_\lambda\leq u_n\ \mbox{for all}\ n\in\NN, \\
	&\Rightarrow &\overline{u}^*_\lambda \leq u^*_\lambda, \\
	&\Rightarrow &u^*_\lambda\in S^*_\lambda\ \mbox{and}\ u^*_\lambda=\inf S^\lambda_+.
	\end{eqnarray*}
	
	From Corollary \ref{cor14}, we infer the strict monotonicity of the map $\lambda\mapsto u^*_\lambda$.
	
	Finally, suppose that $\{\lambda_n,\lambda\}_{n\geq1}\subseteq(0,\lambda^*)$ and $\lambda_n\rightarrow \lambda^-$. Then
	\begin{eqnarray*}
		&&u^*_{\lambda_n}\leq \tilde{u}_{\lambda^*}\ \mbox{for all}\ n\in\NN\ \mbox{(see the proof of Proposition \ref{prop11})}, \\
		&\Rightarrow &\{u^*_{\lambda_n}\}_{n\geq1}\subseteq W^{1,p}(\Omega)\ \mbox{is bounded}.
	\end{eqnarray*}
	
	From Lieberman \cite{23}, we know that there exist $\alpha\in(0,1)$ and $M_5>0$ such that
	\begin{equation*}
		u_n\in C^{1,\alpha}(\overline{\Omega})\ \mbox{and}\ ||u_n||_{C^{1,\alpha}(\overline{\Omega})} \leq M_5\ \mbox{for all}\ n\in\NN.
	\end{equation*}
	
	Exploiting the compact embedding of $C^{1,\alpha}(\overline{\Omega})\ \mbox{into}\ C^1(\overline{\Omega})$ we have
	\begin{equation}\label{eq106}
		u^*_{\lambda_n}\rightarrow\tilde{u}^*_\lambda\ \mbox{in}\ C^1(\overline{\Omega})	
	\end{equation}
	(here we have the original sequence since it is increasing).
	
	Suppose that $\tilde{u}^*_\lambda\neq u^*_\lambda$. Then we can find $z_0\in\Omega$ such that
	\begin{eqnarray*}
		&&u^*_\lambda(z_0)<\tilde{u}^*_\lambda(z_0), \\
		&\Rightarrow & u^*_\lambda(z_0)<u^*_{\lambda_n}(z_0)\ \mbox{for all}\ n\geq n_0\ \mbox{(see (\ref{eq106})).}
	\end{eqnarray*}
	
	This contradicts the monotonicity of $\lambda\mapsto u^*_\lambda$. Therefore $\tilde{u}^*_\lambda=u_\lambda$ and the map $\lambda\mapsto u^*_\lambda$ is left continuous.
\end{proof}

	\medskip
{\bf Acknowledgments.} This research was supported by the Slovenian Research Agency grants P1-0292, J1-8131, J1-7025. V.D. R\u adulescu acknowledges the support through a grant of the Romanian National Authority for Scientific Research and Innovation, CNCS-UEFISCDI, project number PN-III-P4-ID-PCE-2016-0130.

\end{document}